\documentclass[a4paper,11pt]{article}
\usepackage[margin=2.5cm]{geometry}

\usepackage[T1]{fontenc}
\usepackage[utf8]{inputenc}
\usepackage{url}
\usepackage{graphicx}
\usepackage{mathtools}
\usepackage{amsmath}
\usepackage{amssymb}
\usepackage{amsthm}
\usepackage{dsfont}
\usepackage{amsfonts}
\usepackage{mathrsfs}
\usepackage{bbm}
\usepackage{enumerate}
\usepackage{graphicx,tikz,pgf}
\usepackage[pdfborder={0 0 0}]{hyperref}
\newcounter{NET}

\usetikzlibrary{arrows}

\usepackage[all,cmtip]{xy}
\usepackage{amssymb,amsmath,amsthm,mathrsfs}
\def\RR{{\mathbb R}}
\def\CC{{\mathbb C}}
\def\NN{{\mathbb N}}
\def\ZZ{{\mathbb Z}}

\def\A{{\mathcal A}}
\def\B{{\mathcal B}}

\def\D{{\mathcal D}}

\def\H{{\mathcal H}}
\def\I{{\mathcal I}}
\def\K{{\mathcal K}}

\def\U{{\mathcal U}}
\def\V{{\mathcal V}}
\def\W{{\mathcal W}}

\def\a{\alpha}
\def\b{\beta}

\def\e{\varepsilon}

\def\g{\gamma}

\def\i{\iota}

\def\t{\tau}

\def\z{\zeta}

\def\Ad{{\hbox{\rm Ad\,}}}
\def\Aut{{\hbox{Aut}}}

\def\End{{\hbox{End}}}

\def\sp{{\rm sp}\,}

\def\1{{\mathbbm 1}}
\def\Exp{{\rm Exp}}

\def\diff{{\rm Diff}_+}

\def\diffs1{\diff(S^1)}

\def\vect{{\rm Vect}}

\def\vir{{\rm Vir}}

\def\supp{{\rm supp\,}}

\def\psl2r{{\rm PSL}(2,\RR)}
\def\sl2r{{\rm SL}(2,\RR)}
\def\su11{{\rm SU}(1,1)}
\def\2dmob{{\overline{\psl2r}\times\overline{\psl2r}}}
\def\<{\langle}
\def\>{\rangle}

\def\sint{\lfloor s\rfloor}

\def\dom{{\mathscr{D}}}

\newcommand{\fin}{\mathrm{fin}}

\newtheorem{theorem}{Theorem}[section]
\newtheorem{lemma}[theorem]{Lemma}
\newtheorem{corollary}[theorem]{Corollary}
\newtheorem{proposition}[theorem]{Proposition}

\theoremstyle{definition}

\theoremstyle{remark}
\newtheorem{remark}[theorem]{Remark}

\numberwithin{equation}{section}

\title{Positive energy representations of Sobolev diffeomorphism groups of the circle}
\date{}
\author{
 {\bf Sebastiano Carpi}\footnote{Supported in part by ERC advanced grant 669240 QUEST ``Quantum Algebraic Structures and Models'' and GNAMPA-INDAM.}\\ 
Dipartimento di Matematica, Universit\`a di Roma ``Tor Vergata''\\
   Via della Ricerca Scientifica 1, I-00133 Roma, Italy\\
   email: {\tt carpi@mat.uniroma2.it}  \\ 
  \\
{\bf Simone Del Vecchio}\footnote{Supported by ERC advanced grant 669240 QUEST ``Quantum Algebraic Structures and Models''.}
\\
   Dipartimento di Matematica, Universit\`a di Roma ``Tor Vergata''\\
   Via della Ricerca Scientifica 1, I-00133 Roma, Italy\\
   email: {\tt delvecch@mat.uniroma2.it}\\
\\
{\bf Stefano Iovieno}
\\
   Dipartimento di Matematica, Universit\`a di Roma ``La Sapienza''\\
   Piazzale Aldo Moro 5, I-00185 Roma, Italy\\
   email: {\tt iovieno@mat.uniroma1.it}\\
\\
{\bf Yoh Tanimoto}\footnote{Supported until February 2020 by Programma per giovani ricercatori, anno 2014 ``Rita Levi Montalcini''
of the Italian Ministry of Education, University and Research.}
\\
  Dipartimento di Matematica, Universit\`a di Roma ``Tor Vergata''\\
   Via della Ricerca Scientifica 1, I-00133 Roma, Italy\\
   email: {\tt hoyt@mat.uniroma2.it} 
      \vspace{0.5cm}
}

\begin{document}

\maketitle
\begin{abstract}
We show that any positive energy projective unitary representation of
$\diff(S^1)$ extends to a strongly continuous projective unitary
representation of the fractional Sobolev diffeomorphisms $\D^s(S^1)$ for any real $s>3$,
and in particular to $C^k$-diffeomorphisms $\diff^k(S^1)$ with $k\geq4$. A similar result holds for the
universal covering groups provided that the representation is assumed
to be a direct sum of irreducibles.

As an application we show that a conformal net of von Neumann algebras on $S^1$ is covariant with respect to $\D^s(S^1)$, $s > 3$.
Moreover every direct sum of irreducible representations of a conformal net is also $\D^s(S^1)$-covariant. 
\end{abstract}

\section{Introduction}
The group of (smooth) diffeomorphisms of a manifold has been extensively studied and 
there have been many interesting results concerning its algebraic and topological properties, see e.g.\! \cite{Milnor84}.
Among them, the group $\diff(S^1)$ of orientation preserving diffeomorphisms of the circle $S^1$ is of particular interest
in connection with conformal field theory.
In $(1+1)$-dimensional conformal field theory, the symmetry group of the chiral components is $\diff(\RR)$
and often this can be extended to $\diff(S^1)$.
As this group contains spacetime translations,
the relevant representations must be {\it positive energy representations}
and they act on the space of local observables.
The representation theory of positive energy representations has been exploited for construction and classification of
a certain subclass of conformal field theories, see e.g.\! \cite{KL04-1}. 

Non-trivial positive energy representations of $\diff(S^1)$ are necessarily projective.
Any irreducible unitary positive energy representation of the Virasoro algebra extends to a projective representation of the Lie algebra $\vect(S^1)$,
the Lie algebra of vector fields on $S^1$, and it integrates to a positive energy projective unitary representation of 
$\diff(S^1)$ \cite{Neretin83, GW85, Toledano-Laredo99-1}. It follows from \cite[Theorem A.2]{Carpi04}, see also \cite[Section 3.2]{CKLW18}, 
that all irreducible positive energy unitary projective representations of $\diff(S^1)$ arise in this way.
Accordingly they are completely classified by  the central charge $c$ and the lowest conformal energy $h$ \cite{KR87}.
Related results including reducible representations have been recently obtained in \cite{NS15,Zellner17}.

These representations of $\vect(S^1)$ extend to certain non-smooth vector fields
as linear maps \cite{CW05}.
Apart from that this fact had many applications (e.g.\! 
the uniqueness of conformal covariance in conformal nets \cite{CW05},
positivity of energy in DHR sectors \cite{Weiner06}, split property in conformal nets \cite{MTW18} and covariance of soliton representations \cite{Henriques19, DIT19}),
it leads naturally to the question whether the group representations
extend to suitable groups of non-smooth diffeomorphisms.
In contrast to the wide range of results and applications
concerning the algebraic, analytic and topological properties of the group $\diff^k(M)$ of $C^k$ diffeomorphisms 
and $\D^s(M)$ of Sobolev class diffeomorphisms (see e.g.\! \cite{EM70, Misiolek97, Baynaga97, KW09, Figalli10})
 and some results on (true) representations \cite{KL02, AM06, Kuzmin07, Malliavin08},
there appears to be only few results in the literature on positive energy representations of these groups.
Indeed, $\D^s(M)$ is an infinite-dimensional manifold modelled on
the space $H^s(M)$ of $H^s$-vector fields, which is {\it not} a Lie algebra with
the usual Lie bracket for $\vect^\infty(M)$. This makes the study of representations of $\D^s(M)$ rather subtle.

In this paper, we show that any positive energy (projective) representation of the diffeomorphism group
extends to $\D^s(S^1)$ for $s>3$.  We do this first in the irreducible case by considering the action of $\D^s(S^1)$ on vector fields, and therefore,
by exploiting the representation theory of the Virasoro algebra. To obtain the result for the general (reducible) case, we show that the irreducible projective representations which have the same central charge $c$ can be made locally into multiplier representations with the same cocycle and this allows us to take the direct sum of these projective representations. It turns out that conformal nets are covariant with respect to this extended action.

For some special representations appearing in Fock space,
further extensions have been done first to $C^3$-diffeomorphisms \cite{Vromen13}, then to $\D^s(S^1)$, $s>2$ \cite{DIT19}.
The arguments depend on realizing these representations in some specific conformal field theory,
and it is open whether the results are valid for general central charge $c$.
In contrast, by our argument, representations extend to $\D^s(S^1)$ for any real $s > 3$ and for any $c$.
While the extensions to $\D^s(S^1)$ do not necessarily act nicely on the Lie algebra representations  when $2<s\le 3$,
they do so and are differentiable when $s>3$.

Indeed, our proof follows in part the strategy in \cite{GW85} for the integrability of the representations of the Virasoro algebra.
The extension to non-smooth diffeomorphisms then follows from the above mentioned extension to non-smooth vector fields of
the corresponding projective representation of $\vect(S^1)$ given in \cite{CW05}.
Actually, our argument can be used to give a simpler proof of the results in \cite{GW85}, see Remark \ref{remarkGW}.

This paper is organized as follows.
In Section \ref{preliminaries}, we recall the relevant groups and algebras, their topologies
and representations.
In Section \ref{extension}, we first extend the irreducible projective representations of $\diff(S^1)$ to $\D^s(S^1)$ with $s>3$.
Then we lift them locally to multiplier representations, and show that the direct sum can make sense
as projective representations.
Section \ref{conformal} demonstrates that two-dimensional chiral conformal field theories described by
conformal nets of von Neumann algebras have
this extended symmetry of $\D^s(S^1)$.
We summarize possible further continuation of this work in Section \ref{outlook}.

\section{Preliminaries}\label{preliminaries}
\subsection{\texorpdfstring{$\diff(S^1)$}{diffs1} and the Virasoro algebra}
\paragraph{The diffeomorphism group.}
Let us denote by $\diff(S^1)$ the group of orientation preserving, smooth diffeomorphisms of the circle $S^1\coloneqq \lbrace z\in\CC :\vert z\vert=1\rbrace$ and $\vect(S^1)$ denote the set of smooth real vector fields on $S^1.$
$\diff(S^1)$  is an infinite dimensional Lie group whose Lie algebra is identified with the real topological vector space
$\vect(S^1)$ of smooth vector fields on $S^1$ with $C^\infty$ topology \cite{Milnor84}.
In the following we identify $\vect(S^1)$ with $C^\infty(S^1,\mathbb{R})$ and
for $f\in C^{\infty}(S^1,\mathbb{R})$ we denote by $f^\prime$ the derivative of $f$ with respect to the angle $\theta$,
$$ f^\prime(z)=\frac{d}{d\theta}f(e^{i\theta})\bigg\rvert_{e^{i \theta}=z}.$$
We consider a diffeomorphism $\gamma\in\diff(S^1)$ as a map from $S^1$
to $S^1 \subset \CC$. With this convention, its action on $f\in \vect(S^1)$ is 
\begin{align}\label{eq:defgammaf}
(\gamma_* f)(e^{i\theta})
=-ie^{-i\theta}\left(\frac{d}{d\varphi}\gamma(e^{i\varphi})\right)\bigg\rvert_{e^{i\varphi} = \gamma^{-1}(e^{i\theta})}f(\gamma^{-1}(e^{i\theta})). 
\end{align}

We denote by $\diff^k(S^1)$ the group of $C^k$-diffeomorphisms of $S^1$.
Note that this is not a Lie group, and indeed, the corresponding linear space $\vect^k(S^1)$
of $C^k$-vector fields is not closed under the natural Lie bracket (see below).

The universal covering group of $\diff(S^1)$ (resp.\! $\diff^k(S^1)$), $\widetilde{\diff(S^1)}$ (resp.\! $\widetilde{\diff^k(S^1)}$),
can be identified\footnote{The realization of $\widetilde{\diff^k(S^1)}$ works in the same way as $\widetilde{\diff(S^1)}$
as in \cite[Section 6.1]{Toledano-Laredo99-1}, see also \cite[Example 4.2.6]{Hamilton82}.} with the group of $C^{\infty}$-diffeomorphisms
(resp.\! $C^k$-diffeomorphisms) $\gamma$ of $\mathbb{R}$ which satisfy
\begin{equation*}
\gamma(\theta+2\pi)=\gamma(\theta)+2\pi.
\end{equation*}
If $\gamma\in\widetilde{\diff(S^1)}$, its image under the covering map is in the following denoted by $\mathring{\gamma}\in\diff(S^1)$, where $\mathring{\gamma}(e^{i\theta})=e^{i\gamma(\theta)}$.
Conversely, if $\g \in \diff(S^1)$, there is an element $\tilde{\g} \in \widetilde{\diff(S^1)}$
whose image under the covering map is $\g$. Such a $\tilde{\g}$ is unique up to $2\pi$
and called a lift of $\g$.

The group $\diff(S^1)$ admits the Bott-Virasoro cocycle $B:\diff(S^1)\times\diff(S^1)\rightarrow \mathbb{R}$ (see e.g. \cite{FH05}).
The Bott-Virasoro group is then defined as the group with elements
\[
 (\gamma, t)\in\diff(S^1)\times\mathbb{R}
\]
and with multiplication
\[
 (\gamma_1,t_1)\cdot(\gamma_2,t_2)=(\gamma_1\gamma_2, t_1+t_2+ B(\gamma_1,\gamma_2)).
\]
Note that, given a true (not projective) unitary irreducible representation $V$ of the universal covering of
the Bott-Virasoro group, one can obtain a unitary multiplier representation\footnote{for the definition of unitary multiplier representation see Section \ref{projective}.} $\underline{V}(\gamma) := V(\gamma, 0)$ of $\widetilde{\diff(S^1)}$
(with respect to the Bott-Virasoro cocycle $B$). Then the map $\underline{V}:\widetilde{\diff(S^1)}\rightarrow U(\mathcal{H})$ satisfies
\[
 \underline{V}(\gamma_1)\underline{V}(\gamma_2)=e^{ic B(\mathring{\gamma_1},\mathring{\gamma_2})}\underline{V}(\gamma_1 \gamma_2),
\]
where $c\in \RR$ by irreducibility.

\paragraph{The Lie algebra.}
The space $\vect(S^1)$ is endowed with the Lie algebra structure with
the Lie bracket given by
\[
 [f,g]=f^{\prime}g-f g^{\prime}.
\]

As a Lie algebra, $\vect(S^1)$ admits the Gelfand–Fuchs two-cocycle 
\begin{equation}\label{GFcocycle}
\omega (f,g)=\frac{1}{48\pi}\int_{S^1}(f(e^{i\theta})g^{\prime\prime\prime}(e^{i\theta})-f^{\prime\prime\prime}(e^{i\theta})g(e^{i\theta}))d\theta.
\end{equation}

The Virasoro algebra $\vir$ is the central extension of the complexification of
the algebra generated by the trigonometric polynomials in $\vect(S^1)$ defined by
the two-cocycle $\omega$.
It can be explicitly described as the complex Lie algebra generated by $L_n$, $n\in\mathbb{Z}$,
and the central element $\mathfrak{\kappa}$, with brackets
\[
 [L_n,L_m]=(n-m)L_{n+m}+\delta_{n+m,0}\frac{n^3-n}{12}\mathfrak{\kappa}.
\]
Consider a representation $\rho:\vir\rightarrow\End(V)$ of $\vir$ on a complex vector space $V$
endowed with a scalar product $\langle\cdot,\cdot\rangle$. We call $\rho$ a {\bf unitary positive energy representation} if the following hold
\begin{enumerate}
\item Unitarity: $\langle v,\rho(L_n)w\rangle=\langle \rho(L_{-n})v,w\rangle$ for every $v,w\in V$ and $n\in\ZZ$;
\item Positivity of the energy: $V=\bigoplus_{\lambda\in\RR_+\cup\lbrace 0\rbrace}V_{\lambda}$, where $V_{\lambda}\coloneqq \ker(\rho(L_0)-\lambda\1_V)$. The lowest eigenvalue of $\rho(L_0)$ is called lowest weight;
\item  Central charge: $\rho(\mathfrak{\kappa})=c\1_V$;
\end{enumerate}
There exists an irreducible unitary positive energy representation with central charge $c$ and lowest weight $h$ if and only if $c\ge 1$ and $h\ge 0$
(continuous series representation) or $(c,h)=(c(m),h_{p,q}(m))$, where $c(m)=1-\frac{6}{(m+2)(m+3)}$, $h_{p,q}(m)=\frac{(p(m+1)-qm)^2-1}{4m(m+1)}$, $m=3,4,\cdots$, $p=1,2,\cdots,m-1$, $q=1,2,\cdots,p$, (discrete series representation) \cite{KR87}\cite{DMS97}. In this case the representation space $V$ is denoted by $\H^\fin(c,h)$. We denote by $\mathcal{H}(c,h)$ the Hilbert space completion of the vector space $\H^\fin(c,h)$ associated with the unique irreducible unitary positive energy representation of $\vir$
with central charge $c$ and lowest weight $h$.

In these representations, the conformal Hamiltonian
$\rho(L_0)$ is diagonalized, and on the linear span of its eigenvectors
$\mathcal{H}^{\fin}(c,h)$ (the space of finite energy vectors),
the Virasoro algebra acts algebraically as unbounded operators.

\paragraph{The stress-energy tensor.}
Let $\mathcal{H}(c,h)$ as above and, with abuse of notation,
we denote by $L_n$ the elements of $\vir$ represented in $\mathcal{H}(c,h)$.
For a smooth complex-valued function $f$ on $S^1$ with finitely many non-zero Fourier components,
the (chiral) stress-energy tensor associated with $f$ is the operator
$$T(f)=\sum_{n\in\mathbb{Z}}L_n \hat{f}_n$$
acting on $\mathcal{H}(c,h)$, where
$$\hat{f}_n=\int_0^{2\pi}\frac{d\theta}{2\pi}e^{-in\theta}f(e^{i\theta}).$$
The stress-energy tensor $T$ can be extended to a particular linear space of functions strictly containing the set of all smooth functions,
and when $f$ is a real-valued function, $T(f)$ is essentially self-adjoint on $\mathcal{H}^{\fin}(c,h)$ \cite{CW05}.
This fact will be used in this article and will be thus resumed in some detail in Section \ref{non-smooth}.

It is a crucial fact that the irreducible representations $\mathcal{H}(c,h)$ of $\vir$ integrate to irreducible unitary strongly continuous representations of the universal covering of the Bott-Virasoro group \cite{FH05}.
In other words, denoting by $q$ the quotient map $q: \U(\mathcal{H}(c,h))\rightarrow \U(\mathcal{H}(c,h))/\mathbb{C}$
(we denote by $\U(\K)$ the group of unitary operators on $\K$),
there is an irreducible, unitary, strongly continuous multiplier representation  $U$ of $\widetilde{\diff(S^1)}$,
the universal covering of $\diff(S^1)$, such that
\[
 q(U(\Exp(f)))=q(e^{iT(f)})
\]
for all $f\in\vect(S^1)$, where $\Exp$ is the Lie-theoretic exponential map of $\diff(S^1)$ (see \cite{Milnor84}).

For the stress-energy tensor $T$, we have the following covariance
\cite[Proposition 5.1, Proposition 3.1]{FH05}.
\begin{proposition}\label{pr:covariance}
The stress-energy tensor $T$ on $\mathcal{H}(c,h)$ transforms according to 
\[
 U(\gamma)T(f)U(\gamma)^*=T(\mathring{\gamma}_*({f}))+\frac{c}{24\pi}\int^{2\pi}_0\{\mathring{\gamma},z\}\bigg\rvert_{z=e^{i\theta}}f(e^{i\theta})e^{i2\theta}d\theta
\]
on vectors in $\mathcal{H}^{\fin}(c,h)$, for $f\in\vect(S^1)$ and $\gamma\in\widetilde{\diff(S^1)}$. Furthermore the commutation relations
\[
 i[T(g),T(f)]=T(g^\prime f-f^\prime g)+ c \omega(g,f),
\]
where $\omega$ is the Gelfand–Fuchs two-cocycle \eqref{GFcocycle}, hold for arbitrary $f,g\in C^\infty (S^1)$, on vectors $\psi\in \mathcal{H}^{\fin}(c,h).$ 
\end{proposition}

Here 
\[
 \{\mathring{\gamma},z\}=\frac{\frac{d^3}{dz^3}\mathring{\gamma}(z)}{\frac{d}{dz}\mathring{\gamma}(z)}-\frac{3}{2}\left(\frac{\frac{d^2}{dz^2}\mathring{\gamma}(z)}{\frac{d}{dz}\mathring{\gamma}(z)}\right)^2
\]
is the Schwarzian derivative of $\mathring{\gamma}$ and $\frac{d}{dz}\mathring{\gamma}(z)=-i\bar{z}\frac{d}{d\theta}\mathring{\gamma}(e^{i\theta})\bigg\rvert_{e^{i\theta}=z}$.
Note that
\[
 \beta(\g,f)\coloneqq \frac{c}{24\pi}\int_{S^1}\{\mathring\gamma,z\}izf(z)dz
\]
and $\omega(\cdot,\cdot)$ are related by
\begin{align}\label{eq:gelfandderivative}
\frac{d}{dt}\beta(\Exp(tf),g)\bigg\rvert_{t=0}=-c\omega(f,g).
\end{align}

\subsection{The stress-energy tensor on non-smooth vector fields}\label{non-smooth}
Let $T$ be the stress-energy tensor on $\H(c,h)$.
Given a not necessarily smooth real function $f$ of $S^1$ it is possible to evaluate the stress-energy tensor on $f$ \cite[Proposition 4.5]{CW05}. First of all we define for a real-valued function $f$ of the circle
\[
\Vert f\Vert_{\frac{3}{2}}\coloneqq \sum_{n\in\mathbb{Z}}\vert{\hat{f}}_n\vert(1+|n|^{\frac{3}{2}}).
\]
We denote\footnote{We consider $\mathcal{S}_{\frac32}(S^1)$ and $H^s(S^1)$ below as the spaces of nonsmooth
vector fields on $S^1$, and accordingly, without specification, they are the spaces of real functions.}
with $\mathcal{S}_{\frac{3}{2}}(S^1)$ the class of functions $f\in L^1(S^1,\RR)$ such that
$\Vert f\Vert_{\frac{3}{2}}$ is finite, endowed with the topology induced by the norm $\Vert\cdot\Vert_{\frac{3}{2}}$.

The following is \cite[Proposition 4.2, Theorem 4.4, Proposition 4.5]{CW05}.
\begin{proposition}\label{pr:nonsmooth}
If $f:S^1\rightarrow\mathbb{C}$ is continuous and such that $\sum_{n\in\mathbb{Z}}|\hat{f}_n|(1+|n|^{\frac{3}{2}})<\infty$ then
\begin{enumerate}[{(}1{)}]
\item\label{pr:nonsmooth-def} the operator $T(f)=\sum_{n\in\mathbb{Z}}L_n \hat{f}_n$ on the domain $\mathcal{H}^{\fin}(c,h)$ is well defined, (i.e. the sum is strongly convergent on the domain).
\item\label{pr:nonsmooth-star} $T(f)^*$ is an extension of the operator $T(f)^+:=\sum_{n\in\mathbb{Z}}L_n \bar{\hat f}_{-n}$
(this is again understood as an operator on the domain $\mathcal{H}^{\fin}(c,h)$).
\item\label{pr:nonsmooth-symmetry} $T(f)$ is closable and $\overline{T(f)}=(T(f)^+)^*$, where $T(f)$ and 
$T(f)^+$
are considered as operators on the domain $\mathcal{H}^{\fin}(c,h)$.
In particular, if $\hat{f}_n=\bar{\hat f}_{-n}$ for all $n\in\mathbb{Z}$ (i.e. if $f$ is a real-valued function),
then $T(f)$ is essentially self-adjoint on $\mathcal{H}^{\fin}(c,h)$.
\item\label{pr:nonsmooth-bound} For every $\xi \in \dom(L_0)$ we have the following energy bounds
\[
 \|T(f)\xi\|\leq r\|f\|_{\frac{3}{2}}\|(1+L_0)\xi\|,
\]
where $r$ is a function of the central charge $c$ only.
Consequently, $\dom(L_0) \subset \dom(\overline{T(f)})$.
\item\label{pr:nonsmooth-convergence} If $\{f_n\}$ ($n\in\mathbb{N}$)
is a sequence\footnote{This should be distinguished from the Fourier coefficients $\hat f_n$ of a single function $f$.}
of continuous real functions on $S^1$ in $\mathcal{S}_{\frac{3}{2}}(S^1)$
and $\|f-f_n\|_{\frac{3}{2}}$ converges to $0$ as $n$ tends to $\infty$, then
\[
 \overline{T(f_n)}\rightarrow \overline{T(f)}
\]
in the strong resolvent sense.
\end{enumerate}
\end{proposition}

Hereafter, we denote the closure by the same symbol $T(f)$ as long as this does not cause confusions.

The class $\mathcal{S}_{\frac{3}{2}}(S^1)$ contains many non-smooth functions which are useful in applications, e.g. differentiable functions which are piecewise smooth \cite[Lemma 2.2]{Weiner06},\cite[Lemma 5.3]{CW05}:
\begin{proposition}
If a real-valued function $f$ on the circle is 
piecewise smooth and once continuously differentiable on the whole $S^1$, then
$f \in \mathcal{S}_{\frac{3}{2}}(S^1)$.
\end{proposition}

\subsection{Groups of diffeomorphisms of Sobolev class \texorpdfstring{$H^s(S^1)$}{Hs(S1)}}\label{sobolev}
We introduce (see \cite[Section 2]{EK14} and \cite[Definition 2.2]{EK14}, respectively)
\begin{itemize}
\item for $s\in\mathbb{R}$, $s\geq 0$,
\begin{align*}
  H^s(S^1) &:= \{f\in L^2(S^1, \RR): \|f\|_{H^s} < \infty\}, \text{ where }
 \|f\|_{H^s} := \left(\sum_{n\in\ZZ} (1+n^2)^s|\hat f_n|^2\right)^\frac12, \\
 H^s(S^1,\CC) &:= \{f\in L^2(S^1, \CC): \|f\|_{H^s} < \infty\}, \text{ where }
 \|f\|_{H^s} := \left(\sum_{n\in\ZZ} (1+n^2)^s|\hat f_n|^2\right)^\frac12, 
  \end{align*}
  which we consider as a Banach space (in fact a Hilbert space) with norm $\|\cdot\|_{H^s}$;
\item for $s\in\mathbb{R}$, $s>\frac{3}{2}$,
\begin{equation*}
\D^s(S^1) := \{\g \in \diff^1(S^1): \tilde \g - \i \in H^s(S^1)\},
\end{equation*}
\end{itemize}
where $\tilde \g$ is a lift of $\g$ to $\RR$ and $\i: \mathbb{R} \to \mathbb{R}$ is the identity map.  Here we are identifying 
the $2\pi$- periodic functions  $\tilde \g - \i$ with real valued functions on $S^1 \simeq \mathbb{R} /2\pi \mathbb{Z}$. 

$\D^s(S^1)$ has the structure of a Hilbert manifold modelled  on $H^s(S^1)$, see \cite{EK14,EM70}. It turns out to be a topological group, see 
Lemma \ref{lm:sobolevgroup} below (but not a Lie group).

Actually, in the literature there are various definitions of these Sobolev spaces/manifolds
and their topologies.
Although it is well-known that they coincide, for the convenience of the reader
we recall them and show their equivalence in Appendix.

If $s>\frac12$, the space $H^s(S^1)$ is a subspace of $C(S^1, \RR)$.
Furthermore, from these definitions, it is immediate that
$\diff^k(S^1)$ is continuously embedded in $\D^k(S^1)$ for any positive integer $k$.
Conversely, by the Sobolev-Morrey embedding \cite[Proposition 2.2]{IKT13}, it holds that $\D^s(S^1) \hookrightarrow \diff^k(S^1)$
if $s > k+\frac12$.

The first statement of the following is a straightforward adaptation of \cite[Lemma 2.3]{IKT13}. 
One can also find various different elementary proofs,
for example \cite{timur315086, Smyrlis823756}.
The second statement is an adaptation of \cite[Lemma B.4]{IKT13}.
\begin{lemma}\label{lm:sobolevalgebra}
 Let $s > \frac12$. Then
 $H^s(S^1)$ is an algebra and $\|fg\|_{H^s} \le C_s \|f\|_{H^s}\|g\|_{H^s}$.
 If $g \in H^s(S^1)$ and $\inf_\theta (1+g(\theta))> 0$, then $\frac1{1+g} \in H^s(S^1)$.
\end{lemma}

The following is a special case of \cite[Theorem B.2]{IKT13} and an analogue of \cite[Proposition B.7]{IKT13},
see also the Appendix.
According to \cite[P.12]{Kolev13}, Lemma \ref{lm:sobolevgroup}(a) for integer $s$ has been first established in \cite{Ebin68}.
\begin{lemma}\label{lm:sobolevgroup}
 Let $s > \frac 32$.
 Then
 \begin{enumerate}[{(}a{)}]
  \item $(\g,f) \mapsto f\circ \g,\; \D^s(S^1)\times H^s(S^1) \to H^s(S^1)$ is continuous.
  \item $\g \mapsto \g^{-1},\; \D^s(S^1)\to \D^s(S^1)$ is continuous.
  \item  $\D^s(S^1)$ is a topological group.
 \end{enumerate}
\end{lemma}

By applying these results, we get:
\begin{lemma}\label{lm:gamma32sobolev}
The following hold.
\begin{enumerate}[{(}a{)}]
\item For $s>\frac32$, the map
\begin{align*}
\D^{s+1}(S^1)\times H^s(S^1)&\rightarrow H^s(S^1)\\
(\g,f)&\mapsto \g_*(f),
\end{align*}
where $\g_*(f)$ is as in \eqref{eq:defgammaf},
is continuous.
\item For $s>2$, the embedding $H^s(S^1)\hookrightarrow \mathcal{S}_{\frac{3}{2}}(S^1)$ is continuous.
\item For $s>3$, $\beta(\g,f)$ extends continuously to $\g \in \D^s(S^1), f\in L^2(S^1, \RR)$.
\end{enumerate}
\end{lemma}
\begin{proof}
 
 (a) follows from Lemmas \ref{lm:sobolevgroup} and \ref{lm:sobolevalgebra}
 and \eqref{eq:defgammaf}.
 
 (b) is obtained from the following inequality
 $$\sum_{k\neq 0} |\hat{f}_k||k|^{\frac{3}{2}}=\sum_{k\neq 0} |\hat{f}_k| |k|^{2+\epsilon}\frac{1}{|k|^{\frac{1}{2}+\epsilon}}     \leq \sqrt{\sum_{k\neq0} \frac{1}{k^{1+2\epsilon}}}\sqrt{\sum_{k\neq0} |\hat{f}_k|^{2}|k|^{4+2\epsilon}}.$$
 for any $\epsilon>0$.

 (c) Note that, with $s>3$, $\D^{s}(S^1) \ni \gamma \mapsto \{\mathring{\gamma},z\} \in L^2(S^1, \CC)$
 is continuous. To see it, in the definition
 \[
  \{\mathring{\gamma},z\}=\frac{\frac{d^3}{dz^3}\mathring{\gamma}(z)}{\frac{d}{dz}\mathring{\gamma}(z)}-\frac{3}{2}\left(\frac{\frac{d^2}{dz^2}\mathring{\gamma}(z)}{\frac{d}{dz}\mathring{\gamma}(z)}\right)^2,
 \]
 the maps $\g\mapsto\frac{d^3}{dz^3}\mathring{\gamma}(z)\in L^2(S^1, \CC)$ and $\g\mapsto \frac{1}{\frac{d}{dz}\mathring{\gamma}(z)} \in H^{s-1}(S^1, \CC) \subset L^\infty(S^1, \CC)$
 are continuous, hence their product is continuous in $L^2(S^1, \CC)$.
 The second derivative $\g\mapsto\frac{d^2}{dz^2}\mathring{\gamma}(z) \in H^{s-2}(S^1, \CC)$
 is continuous hence so is
 $\g\mapsto\left(\frac{\frac{d^2}{dz^2}\mathring{\gamma}(z)}{\frac{d}{dz}\mathring{\gamma}(z)}\right)^2 \in H^{s-2}(S^1, \CC)$
 (by the complexification of Lemma \ref{lm:sobolevalgebra}),
 hence we obtain the continuity of $\g \mapsto \{\mathring{\gamma},z\}$ by the complexification of
 Lemma \ref{lm:sobolevalgebra}.
 Now the claim is immediate because $\beta(\g, f) = \frac{c}{24\pi}\int_{S^1}\{\mathring\gamma,z\}izf(z)dz.$
\end{proof}

The universal covering group $\widetilde{\D^s(S^1)}$ of $\D^s(S^1)$
is algebraically a subgroup of $\widetilde{\diff^1(S^1)}$, namely the space of the maps $\g:\RR\to\RR$ satisfying
$\gamma(\theta+2\pi)=\gamma(\theta)+2\pi$ and locally $H^s$ (see Appendix),
and this can be identified with an open convex subset of $H^s(S^1)$.

\subsection{Projective and multiplier representations}\label{projective}
A unitary \textbf{multiplier representation} of a topological group $G$ is a pair $(U,\mathcal{H})$ where $U:G\rightarrow \mathcal{U}(\mathcal{H})$ is a map such that $U(g_1)U(g_2)=\sigma(g_1,g_2)U(g_1g_2)$ with $\sigma:G\times G\rightarrow \mathbb{T}$. The map $\sigma$ automatically satisfies the equality
\begin{equation*}
\sigma(g_1,g_2)\sigma(g_1g_2,g_3)=\sigma(g_1,g_2g_3)\sigma(g_2,g_3).
\end{equation*}
A unitary multiplier representation $U$ of $G$ is continuous in the strong operator topology (SOT)
if $U(g)v$ tends to $U(g_0)v$ for all $v\in\H$ if $g$ tends to $g_0$.

A SOT continuous unitary projective representation of a topological group $G$ is a pair
$(U,\mathcal{H})$ where $\mathcal{H}$ is a 
Hilbert space and $U$ is a continuous group homomorphism from $G$
to $\mathcal{U}(\mathcal{H})/\mathbb{T}$, where $\U(\H)$ is equipped with the SOT
and $\mathcal{U}(\mathcal{H})/\mathbb{T}$ with the quotient topology by the quotient map $q$.

Now let $P(\mathcal{H}) = \mathcal{H}/\mathbb{T}$ be the projective space associated to the Hilbert space $\mathcal{H}$ endowed with quotient topology. Then $\U(\mathcal{H})/\mathbb{T}$ acts on $P(\H)$ in a natural way and, as a consequence of \cite[Theorem 1.1]{Bargmann54}, the quotient topology on $\U(\H)/\mathbb{T}$ coincides with SOT on $\U(\H)/\mathbb{T}$ induced by this action. 
Note that every SOT continuous multiplier representation of $G$ on $\H$ gives rise to a SOT continuous projective representation of $G$. Conversely, by  \cite[Theorem 1.1]{Bargmann54},  every SOT continuous projective representation of $G$ gives rise to a continuous local multiplier representation of $G$ defined on a suitable neighborhood of the identity. 
It is well known that a projective unitary representation $U$ is SOT continuous if its action on $\B(\H)$ is pointwise SOT continuous, i.e.\ $g\mapsto U(g)xU(g)^*\xi$ is a continuous map for all $x \in \B(\mathcal{H})$ and all $\xi \in \mathcal{H}$. We outline an argument here for the convenience of the reader. 
It is clear from the above discussion that if $U$ is SOT continuous then it acts pointwise SOT continuously on $\B(\H)$. Let $g_\lambda$ be a net in $G$ converging to the identity, $\xi$ be a unit vector in $\H$ and let $p_\xi$ be the corresponding projection. Then, since $U(g_\lambda)p_\xi U(g_\lambda)^*$ converges in the SOT to $p_\xi$, $\|(U(\g_\lambda)\xi,\xi)\|$ converges to $1$. Since $\xi$ was arbitrary, it follows by \cite[Theorem 1.1]{Bargmann54} that $U$ acts continuously on $P(\H)$ and hence it is SOT continuous.

\section{Extension of the \texorpdfstring{$\diff(S^1)$}{diffs1} representations to Sobolev diffeomorphisms}\label{extension}

\subsection{Irreducible case}
The purpose of this section is to extend the (positive energy projective) representation $U$ on $\H(c,h)$ of $\diff(S^1)$ to $\D^s(S^1)$ with $s>3$.
In the following $s>3$ will always be assumed.

An element $\gamma\in\D^s(S^1)$ acts on $f\in\vect(S^1)$ via \eqref{eq:defgammaf}.
If $T$ is the energy-momentum operator associated with a positive energy unitary representation of the Virasoro algebra
$\vir$ with central charge $c$ and lowest weight $h$, we define a new class of operators
\begin{align*}
T^{\gamma}(f)\coloneqq T(\gamma_*f)-\beta(\gamma,f),
\end{align*}
where $f \in \vect(S^1)$ and $\beta(\gamma,f)=\frac{c}{24\pi}\int_{S^1}\{\gamma,z\}izf(z)dz$,
which makes sense for $\gamma \in \D^s(S^1)$
by Lemma \ref{lm:gamma32sobolev} and Proposition \ref{pr:nonsmooth}(\ref{pr:nonsmooth-def}).
The fact that $\gamma_*f$ is in $\mathcal{S}_{\frac{3}{2}}(S^1)$
ensures that $T(\gamma_*f)$
is an essentially self-adjoint operator on $\mathcal{H}^\fin(c,h)$ and so is $T^{\gamma}(f)$
by Proposition \ref{pr:nonsmooth}(\ref{pr:nonsmooth-symmetry}).
We denote its closure by the same symbol $T^\gamma(f)$, so long as no confusion arises.

Note that, if $\gamma \in \diff(S^1)$, then we have
\begin{align}\label{eq:Tgammasmooth}
 T^{\gamma}(f) = \Ad U(\gamma)(T(f)).
\end{align}
Indeed, by definition $T^{\gamma}(f) = T(\gamma_*f)-\beta(\gamma,f)$
and by Proposition \ref{pr:covariance}, \eqref{eq:Tgammasmooth} holds on $\dom(L_0)$,
and both operators are essentially self-adjoint there, hence they must coincide.
Since in the smooth case the transformation $T\rightarrow T^{\gamma}$ is unitarily implemented, the energy bound holds as well:
\begin{align}\label{eq:smoothbound}
\|T^{\gamma}(f)\xi\|
\le r\|f\|_{\frac{3}{2}}\cdot \|(1+L_0^{\gamma})\xi\|,
\end{align}
where $L_0^{\gamma} := T^{\gamma}(1)$.

We define for $\g_1,\g_2\in\D^s(S^1)$
\[
(T^{\g_1})^{\g_2}(f)\coloneqq T^{\gamma_1}((\gamma_2)_*f)-\beta(\gamma_2,f).
\]
\begin{proposition}\label{lm:composition}
Let $\g_1,\g_2\in\D^s(S^1)$, $s>3$, and $f\in\vect(S^1)$. Then $(T^{\g_1})^{\g_2}(f)=T^{\g_1\g_2}(f)$.
\end{proposition}
\begin{proof}
Using the properties of the Schwarzian derivative \cite{OT05}
\[
\left\{\gamma_1\gamma_2,z\right\}=\left\{\gamma_1,\gamma_2(z)\right\}\left(\frac{d}{dz}\gamma_2(z)\right)^2+\left\{\gamma_2,z\right\},
\]
where $y=\gamma_2(z)$, we infer that
\begin{align*}
\beta(\gamma_1\gamma_2,f)&=-\frac{c}{24\pi}\int_{0}^{2\pi}\left\{\gamma_1\gamma_2,z\right\}\bigg\rvert_{z=e^{i\theta}}f(e^{i\theta})e^{i2\theta}d\theta \\
&=-\frac{c}{24\pi}\int_0^{2\pi}\left\{\gamma_1,y\right\}\bigg\rvert_{y=\g_2(e^{i\theta})}\left(\frac{d}{dz}\gamma_2(z)\right)^2\bigg\rvert_{z=e^{i\theta}}f(e^{i\theta})e^{i2\theta}d\theta\\
&\qquad-\frac{c}{24\pi}\int_0^{2\pi}\left\{\gamma_2,z\right\}\bigg\rvert_{z=e^{i\theta}}f(e^{i\theta})e^{i2\theta}d\theta \\
&=-\frac{c}{24\pi}\int_0^{2\pi}\left\{\gamma_1,y\right\}\bigg\rvert_{y=e^{i\varphi}}\cdot(-i)\frac{d}{d\theta}\left(\g_2(e^{i\theta})\right)\bigg\rvert_{e^{i\theta}=\g_2^{-1}(e^{i\varphi})}f(\gamma_2^{-1}(e^{i\varphi}))e^{i\varphi}d\varphi
\\
&\qquad-\frac{c}{24\pi}\int_0^{2\pi}\left\{\gamma_2,z\right\}\bigg\rvert_{z=e^{i\theta}}f(e^{i\theta})e^{i2\theta}d\theta \\
&=-\frac{c}{24\pi}\int_0^{2\pi}\left\{\gamma_1,y\right\}\bigg\rvert_{y=e^{i\varphi}}\cdot(-i)e^{-i\varphi}\frac{d}{d\theta}\left(\g_2(e^{i\theta})\right)\bigg\rvert_{e^{i\theta}=\g_2^{-1}(e^{i\varphi})}f(\gamma_2^{-1}(e^{i\varphi}))e^{i2\varphi}d\varphi
\\
&\qquad-\frac{c}{24\pi}\int_0^{2\pi}\left\{\gamma_2,z\right\}\bigg\rvert_{z=e^{i\theta}}f(e^{i\theta})e^{i2\theta}d\theta \\
&=\beta(\gamma_1,\g_{2_{*}}(f))+\beta(\gamma_2,f),
\end{align*}
where we used the change of variables $e^{i\varphi} = \g_2(e^{i\theta})$,
hence $e^{i\theta}\frac{d\theta}{d\varphi}\frac{d\g_2}{dz}(e^{i\theta})|_{\g_2(e^{i\theta})=e^{i\varphi}}=e^{i\varphi}$, $\frac{d\g_2}{dz}(e^{i\theta})=-ie^{-i\theta}\frac{d}{d\theta}\g_2(e^{i\theta})$ and \eqref{eq:defgammaf}.

So $(T^{\g_1})^{\g_2}(f)=T((\gamma_1)_*((\gamma_2)_*f))-\beta(\gamma_1,\gamma_{2*}f)-\beta(\gamma_2,f) = T((\gamma_1 \gamma_2)_*f)-\beta(\gamma_1\gamma_2,f)=T^{\gamma_1 \gamma_2}(f)$.
\end{proof}

\begin{lemma}\label{lm:l0gammadomain}
Let $s>3$. $\dom(L_0)=\dom(L_0^{\gamma})$ for every $\gamma\in\D^s(S^1)$.
\end{lemma}
\begin{proof}
By Lemma \ref{lem:localapprox} we can take a sequence $\{\gamma_n\}$ in $\diff(S^1)$ convergent to $\gamma$ in the topology of $\D^s(S^1)$.
We observe that $1 = \lim_n \gamma_{n*}(\gamma^{-1}_*(1))$ in the topology of $\mathcal{S}_{\frac32}(S^1)$
by Lemma \ref{lm:gamma32sobolev}.
For $\xi\in \dom(L_0)$ we know from Proposition \ref{pr:nonsmooth}(\ref{pr:nonsmooth-convergence})
and \eqref{eq:smoothbound} that
\begin{align*}
\|L_0\xi\|&=\lim_{n\to\infty}\|\left(T^{\gamma_n}((\gamma^{-1}_* )(1))+\beta(\g_n,\g_*^{-1}(1))\right)\xi\| \\
&\leq \left(\lim_{n\to\infty} r\|\gamma^{-1}_{*} (1)\|_{\frac{3}{2}}\cdot \|(1+L_0^{\gamma_n})\xi\|+|\beta(\g_n,\g_*^{-1}(1))|\|\xi\|\right)\\
&= r\|\gamma^{-1}_{*} (1)\|_{\frac{3}{2}}\cdot \|(1+L_0^{\gamma})\xi\|+|\beta(\g,\g_*^{-1}(1))|\|\xi\|,
\end{align*}
where the last equality follows again from Lemma \ref{lm:gamma32sobolev}.
Recall that  we know that $\dom(L_0)\subset \dom(L^\gamma_0)$ from Proposition \ref{pr:nonsmooth}(\ref{pr:nonsmooth-bound})
and $L_0^{\gamma}$ is essentially self-adjoint on $\dom(L_0)$.
From the above inequality, we infer that
any sequence $\xi_n \in \dom(L_0)$ converging to $\xi \in \dom(L_0^\gamma)$ in the graph norm of $L_0^\gamma$
is also convergent in the graph norm of $L_0$, and therefore,
we have $\dom(L_0^{\gamma})=\dom(L_0)$.
\end{proof}

\begin{proposition}[energy bounds for $T^\g$]\label{pr:energybound}

Let $\g\in\D^s(S^1)$, $s>3$. Then
$$\Vert T^{\gamma}(f)\xi\Vert \leq r\Vert f\Vert_{\frac{3}{2}}\Vert(1+L_0^{\gamma})\xi\Vert$$
for all $\xi\in \dom(L_0)$.
\end{proposition}
\begin{proof}
Let $\{\gamma_n\}$ a sequence of elements in $\diff(S^1)$ converging to $\gamma\in\D^s(S^1)$ as in Lemma \ref{lem:localapprox}.
By Proposition \ref{pr:nonsmooth}(\ref{pr:nonsmooth-convergence}) and \eqref{eq:smoothbound},
\begin{align*}
\Vert T^{\gamma}(f)\xi\Vert &= \lim_{n\to\infty}\Vert T^{\gamma_n}(f)\xi\Vert\leq \lim_{n\to\infty}r\Vert f\Vert_{\frac{3}{2}}\Vert(1+L_0^{\gamma_n})\xi\Vert=\\
&= r\Vert f\Vert_{\frac{3}{2}}\Vert(1+L_0^{\gamma})\xi\Vert,
\end{align*}
which is the desired inequality.
\end{proof}

\begin{theorem}
Let $\g\in\D^s(S^1)$, $s>3$. $T^{\gamma}$ yields an irreducible unitary positive energy representation of $\vir$ with central charge $c$ and lowest weight $h$ on $\H(c,h)$.
\end{theorem}
\begin{proof}
We are going to prove the Virasoro relations on $C^\infty(L_0^\gamma)$.
For this purpose, we have to take under control the action of various exponentiated operators.

\paragraph{Computations on $\dom(L_0)$.}
Let $f$ and $g$ be real smooth functions. We start by noting that $e^{iT^\gamma(g)} \dom(L_0)\subset \dom(L_0)$.
Indeed, using \cite[Proposition 3.1]{FH05} we have, for $\xi\in \dom(L_0)$ and $\gamma_n \in \diff(S^1)$ as in Lemma \ref{lem:localapprox},
\[
L_0e^{iT^{\gamma_n}(g)}\xi =
e^{iT^{\gamma_n}(g)} ( T((\gamma_n \Exp(-g) \gamma_n^{-1})_*(1)) -
\beta(\gamma_n \Exp(-g) \gamma_n^{-1},1) )\xi,
\]
and the right-hand side converges as $n\rightarrow \infty$ by Proposition \ref{pr:nonsmooth}(\ref{pr:nonsmooth-convergence}).
Therefore, since both $e^{iT^{\gamma_n}(g)}\xi$ and $L_0e^{iT^{\gamma_n}(g)}\xi$ are convergent,
it follows that $e^{iT^{\gamma}(g)}\xi \in \dom(L_0)$ and
\[
L_0e^{iT^{\gamma}(g)}\xi=e^{iT^{\gamma}(g)}( T((\gamma \Exp(-g) \gamma^{-1})_*(1))
- \beta(\gamma \Exp(-g) \gamma^{-1},1) )\xi.
\]

For $\gamma_n \in \diff(S^1)$, by Proposition \ref{pr:covariance}
we have the operator equality
\[
e^{iT^{\gamma_n}(g)}T^{\gamma_n}(f)e^{-iT^{\gamma_n}(g)}=T^{\gamma_n}(\Exp(g)_* (f)) - \frac{c}{24\pi}\int_{S^1}\{\Exp(g),z\}izf(z)dz.
\]

Now, for any positive integer $k$, we consider the function $h_k: \mathbb{R} \to \mathbb{R}$ defined by 
$$h_k(s)= s e^{-\frac{s^2}{k}}.$$ 

Using functional calculus we apply the function $h_k$ to the self-adjoint operators appearing in the two sides of the previous operator equality and we obtain 
\begin{equation}\label{eq:nk}e^{iT^{\gamma_n}(g)}h_k\left(T^{\gamma_n}(f) \right) e^{-iT^{\gamma_n}(g)}=h_k\left(T^{\gamma_n}(\Exp(g)_* (f)) - \frac{c}{24\pi}\int_{S^1}\{\Exp(g),z\}izf(z)dz\right).
\end{equation}

The left-hand side of \eqref{eq:nk} converges strongly to $e^{iT^{\gamma}(g)}h_k(T^{\gamma}(f))e^{-iT^{\gamma}(g)}$
as $n\rightarrow\infty$, because we have convergence of $T^{\gamma_n}(f)$ to $T^{\gamma}(f)$ and
$T^{\gamma_n}(g)$ to $T^{\gamma}(g)$ in the strong resolvent sense and hence, $e^{iT^{\gamma_n}(g)}, h_k(T^{\gamma_n}(f))$ converge
to $e^{iT^{\gamma}(g)}, h_k(T^{\gamma}(f))$, respectively, by \cite[Theorem VIII.20(b)]{RSI}.  Similarly, from the convergence 
of  $T^{\gamma_n}(\Exp(g)_*(f))$ to $T^{\gamma}(\Exp(g)_*(f))$  in the strong resolvent sense it follows that the right-hand side of 
\eqref{eq:nk} converges strongly  to
$$h_k\left(T^{\gamma}(\Exp(g)_* (f)) - \frac{c}{24\pi}\int_{S^1}\{\Exp(g),z\}izf(z)dz\right).$$ Thus 
$$e^{iT^{\gamma}(g)}h_k(T^{\gamma}(f))e^{-iT^{\gamma}(g)}= h_k\left(T^{\gamma}(\Exp(g)_* (f)) - \frac{c}{24\pi}\int_{S^1}\{\Exp(g),z\}izf(z)dz\right)$$ 
so that, if $\xi$ is in  $\dom(L_0)$
$$e^{iT^{\gamma}(g)}h_k(T^{\gamma}(f))e^{-iT^{\gamma}(g)}\xi= h_k\left(T^{\gamma}(\Exp(g)_* (f)) - \frac{c}{24\pi}\int_{S^1}\{\Exp(g),z\}izf(z)dz\right)\xi \, .$$

By taking the limit for $k\rightarrow\infty$, we get for every $\xi\in \dom(L_0)$
\begin{align}\label{eq:commutationexp}
e^{iT^{\gamma}(g)}T^{\gamma}(f)e^{-iT^{\gamma}(g)}\xi=T^{\gamma}(\Exp(g)_* (f))\xi - \left(\frac{c}{24\pi}\int_{S^1}\{\Exp(g),z\}izf(z)dz\right) \xi.
\end{align}

Recall that $\dom(L_0)=\dom(L^{\gamma}_0)$. We get in particular
\begin{align}\label{eq:gammarotation}
e^{itL_0^{\gamma}}T^{\gamma}(f)e^{-itL_0^{\gamma}}\xi = T^{\gamma}(f_t)\xi,
\end{align}
where $f_t(e^{i\theta}) = f(e^{i(\theta -t)})$.

\paragraph{Computations on $C^\infty(L_0^\gamma)$.}
The right-hand side of \eqref{eq:gammarotation}
is differentiable with respect to $t$ when 
$\xi\in \dom(L_0)$ since for the right hand side we get
\[
\lim_{t\rightarrow 0}\frac{1}{t}(T^{\gamma}(f_t)-T^{\gamma}(f))\xi
= \lim_{t\rightarrow 0} T^{\gamma}(\textstyle{\frac1 t}(f_t - f))\xi
= T^{\gamma}(-f^{\prime})\xi = - T^{\gamma}(f^{\prime})\xi,
\]
by the continuity of $T^\gamma$ in the topology of $\mathcal{S}_{\frac32}(S^1)$ (Proposition \ref{pr:energybound}).
Let us specialize it to $\xi\in C^\infty (L_0^\gamma) := \bigcap_n \dom((L^\g_0)^n)$.
For the left-hand side of \eqref{eq:gammarotation}, we have
\begin{align}\label{eq:rotationdiff}
&\left.\frac{d}{dt}\right\vert_{t=0}e^{itL^{\gamma}_0}T^{\gamma}(f)e^{-itL_0^{\gamma}}\xi \nonumber \\
&=\lim_{t\rightarrow\infty}\left(\frac{1}{t}\left(e^{itL^{\gamma}_0}T^{\gamma}(f)e^{-itL_0^{\gamma}}-e^{itL^{\gamma}_0}T^{\gamma}(f)\right)\xi 
+\frac{1}{t}\left(e^{itL^{\gamma}_{0}}T^{\gamma}(f)-T^{\gamma}(f)\right)\xi\right).
\end{align}
The first term converges to $-iT^{\gamma}(f)L_0\xi$. Indeed, by Proposition \ref{pr:energybound},
\begin{align*}
&\left\|\frac{1}{t}\left(e^{itL^{\gamma}_0}T^{\gamma}(f)e^{-itL_0^{\gamma}}-e^{itL^{\gamma}_0}T^{\gamma}(f)\right)\xi+ie^{itL^{\gamma}_0}T^{\gamma}(f)L_0^\gamma\xi\right\| \\
&=\left\|\frac{1}{t}\left(T^{\gamma}(f)e^{-itL_0^{\gamma}}-T^{\gamma}(f)\right)\xi+iT^{\gamma}(f)L_0^\gamma \xi\right \| \\
&\leq r\|f\|_{\frac{3}{2}}\left\|(1+L_0^\gamma)\left(\frac{e^{-itL_0^{\gamma}}-1}{t}+iL^{\gamma}_0\right)\xi\right\| \\
&= r\|f\|_{\frac{3}{2}}\left\|\left(\frac{e^{-itL_0^{\gamma}}-1}{t}+iL^{\gamma}_0\right)(1+L^{\gamma}_0)\xi\right\|.
\end{align*}
Since $\xi\in C^{\infty}(L_0^{\gamma})$, by Stone's theorem \cite[Theorem VIII.7(c)]{RSI}
the above converges to 0 as $t\rightarrow 0$. Thus the limit exists also for the second term
of \eqref{eq:rotationdiff}, and by applying Stone's theorem \cite[Theorem VIII.7(d)]{RSI},
we get $T^{{\gamma}}(f)\xi\in \dom(L^{\gamma}_0)$, and the second term converges to $iL_0^\gamma T^{{\gamma}}(f)\xi$.
or in other words, $T^{{\gamma}}(f)C^\infty(L_0) \subset \dom(L^{\gamma}_0)$
(actually, we proved $T^{{\gamma}}(f)\dom((L_0^\gamma)^2) \subset \dom(L^{\gamma}_0)$).
Thus we have established the following commutation relation on $C^{\infty}(L^{\gamma}_0)$:
\begin{align}\label{eq:commutation}
[L^{\gamma}_0,T^{\gamma}(f)]\xi = iT^{\gamma}(f^{\prime})\xi.
\end{align}

It follows that $C^\infty (L^{\gamma}_0)$ is an invariant domain for every $T^{\gamma}(f)$ with $f\in C^\infty(S^1,\mathbb{R})$. Indeed, for $T^{\gamma}(f)\xi$, with $\xi\in C^{\infty}(L_0^{\gamma})$ and $f\in C^{\infty} (S^1,\mathbb{R})$,
\eqref{eq:commutation} is equivalent to
\begin{align}\label{eq:commutation2}
L^{\gamma}_0 T^{\gamma}(f)\xi= [L^{\gamma}_0,T^{\gamma}(f)]\xi + T^{\gamma}(f)L_0^{\gamma}\xi = iT^{\gamma}(f^{\prime})\xi+ T^{\gamma}(f)L_0^{\gamma}\xi.
\end{align}
We now show that $T^\gamma(f)\xi\in \dom((L_0^\gamma)^{k})$ for every positive integer $k$, using induction on $k$. Assume that $T^{\gamma}(f)\xi\in \dom((L_0^\gamma)^k)$ and all $f\in C^{\infty}(S^1,\mathbb{R})$.
It then follows from \eqref{eq:commutation2} that $L_0^{\gamma}T^{\gamma}(f)\xi\in \dom((L_0^\gamma)^k)$, i.e. $T^\gamma(f)\xi\in \dom((L_0^\gamma)^{k+1})$.
We thus get the desired claim $T^{\gamma}(f)C^{\infty}(L^{\gamma}_0)\subset  C^{\infty}(L^{\gamma}_0)$.

\paragraph{The Virasoro relations.}
Finally we show that the stress-energy tensor $T^{\gamma}$ indeed yields a representation of $\vect(S^1)$. 
For $\xi \in C^\infty(L_0^\gamma)$,
\begin{align}\label{eq:commutationdiff}
&\left.\frac{d}{dt}\right\vert_{t=0}e^{itT^{\gamma}(g)}T^{\gamma}(f)e^{-it T^{\gamma}(g)}\xi \nonumber \\
&=\lim_{t\rightarrow 0}\left(\frac{1}{t}\left(e^{itT^{\gamma}(g)}T^{\gamma}(f)e^{-it T^{\gamma}(g)}-e^{itT^{\gamma}(g)}T^{\gamma}(f)\right)
+\frac{1}{t}\left(e^{itT^{\gamma}(g)}T^{\gamma}(f)-T^{\gamma}(f)\right)\right)\xi.
\end{align}
As for the left-hand side, from \eqref{eq:commutationexp},
we obtain $(T^\gamma(g'f-gf') + c\omega(g,f))\xi$
by \eqref{eq:gelfandderivative}.

Let us see the right-hand side of \eqref{eq:commutationdiff} term by term.
As for the first term, we have
\begin{align}\label{eq:comm1}
&\left\Vert\frac{1}{t}\left(e^{itT^{\gamma}(g)}T^{\gamma}(f)e^{-it T^{\gamma}(g)} - e^{itT^{\gamma}(g)}T^{\gamma}(f)\right)\xi + e^{itT^{\gamma}(g)}\cdot iT^{\gamma}(f)T^{\gamma}(g)\xi\right\Vert \nonumber \\
&=\left\Vert\frac{1}{t}\left(T^{\gamma}(f)e^{-it T^{\gamma}(g)}-T^{\gamma}(f)\right)\xi + iT^{\gamma}(f)T^{\gamma}(g)\xi\right\Vert \nonumber \\ 
&\le r\Vert f\Vert_{\frac{3}{2}}\left\Vert(1+L^{\gamma}_0)\frac{1}{t}\left(e^{-itT^{\gamma}(g)}-1\right)\xi + (1+L^{\gamma}_0)\cdot iT^{\gamma}(g)\xi\right\Vert \nonumber \\
&\le r\Vert f\Vert_{\frac{3}{2}}\left(\left\Vert \left(\frac{1}{t}\left(e^{-itT^{\gamma}(g)}-1\right)+ iT^{\gamma}(g)\right)\xi\right\Vert +\left\Vert\left( \frac{1}{t}L_0^\gamma\left(e^{-itT^{\gamma}(g)}-1\right) + iL_0^\gamma T^{\gamma}(g)\right)\xi\right\Vert\right).
\end{align}
The first term of \eqref{eq:comm1} goes to $0$ by Stone's theorem \cite[Theorem VIII.7(c)]{RSI}.
The second term can be treated by \eqref{eq:commutationexp} and \eqref{eq:commutation} as follows:
\begin{align*}
&\left\Vert \frac{1}{t}L^{\gamma}_0(e^{-itT^{\gamma}(g)}-1)\xi + iL^{\gamma}_0 T^{\gamma}(g)\xi\right\Vert \\
&=\left\Vert\frac{1}{t}\left(e^{-itT^{\gamma}(g)}(T^{\gamma}(\Exp(tg)_* (1))-\beta(\Exp(tg),1))-L^{\gamma}_0\right)\xi + i(iT^{\gamma}(g^{\prime})+T^{\gamma}(g)L^{\gamma}_0)\xi\right\Vert\\
&\leq \left\Vert\frac{1}{t}(e^{-itT^{\gamma}(g)}T^{\gamma}(\Exp(tg)_* (1))-e^{-itT^{\gamma}(g)}L_0^{\gamma})\xi-T^{\gamma}(g^{\prime})\xi\right\Vert \\
&\qquad\qquad+\left\Vert\frac{1}{t}(e^{-itT^{\gamma}(g)}L^{\gamma}_0-L^{\gamma}_0)\xi + iT^{\gamma}(g)L^{\gamma}_0\xi\right\Vert+\bigg\vert\frac{1}{t}\beta(\Exp(tg),1)\bigg\vert\Vert\xi\Vert.
\end{align*}
Each term can be seen to converge to $0$:
the first term is done by noting that $L_0^\gamma = T^\gamma(1)$,
continuity of $T^\g$ (Proposition \ref{pr:energybound}), $[g,1] = g'$ and unitarity of $e^{-itT^{\gamma}(g)}$.
The second term vanishes by using Stone's theorem.
The last term also converges to zero by \eqref{eq:gelfandderivative} and using the fact that $\omega(g,1)=0$.
To summarize, the first term of the right-hand side of \eqref{eq:commutationdiff}
tends to $-iT^\g(f)T^\g(g)$.

The second term of \eqref{eq:commutationdiff} is equal to $iT^{\gamma}(g)T^{\gamma}(f)$. Indeed,
since $C^{\infty}(L^{\gamma}_0)$ is invariant under the action of $T^{\gamma}(f)$, this follows by Stone's theorem.

Altogether, we obtained the equality $i[T^\gamma(g),T^\gamma(f)] = T^\gamma(g'f-gf') + c\omega(g,f)$ on $C^\infty(L_0^\gamma)$,
which is the Virasoro commutation relation.

Note that until here we have only used that $T$ is a positive energy representation
of the Virasoro algebra with central charge $c$ with diagonalizable $L_0$, but not irreducibility.
Therefore, one can iterate our construction for another element in $\D^s(S^1)$.
In particular, by taking $\gamma^{-1}$, we obtain by Proposition \ref{lm:composition}
\begin{equation}\label{eq:gammagamma-1}
 (T^\gamma)^{\gamma^{-1}}(f)= T(f).
\end{equation}

We claim that the new representation $T^\gamma$ is irreducible and has the same lowest weight $h$.
Indeed, by \eqref{eq:gammagamma-1},
one can approximate $T(f)$ by $T^\gamma(\gamma^{-1}_{n*}f)+\beta(\g,(\g_n^{-1})_*(f))$ in the strong resolvent sense,
where $\{\gamma_n\} \subset \diff(S^1)$ and $\gamma_n \to \gamma$ in the topology of $\D^s(S^1)$.
As $\{e^{iT(f)}: f\in\vect(S^1)\}$ generates $\B(\H(c,h))$,
so does $\{e^{iT^\gamma(f)}: f\in\vect(S^1)\}$, and this shows that
$T^\gamma$ is an irreducible representation of the Virasoro algebra.
Furthermore, 
the new conformal Hamiltonian $L^{\gamma}_0=T^{\gamma}(1)$ has spectrum
which is a subset of the spectrum of the old conformal Hamiltonian $L_0$
since it is obtained as a limit in the strong resolvent sense of $\{\Ad U(\gamma_n)(L_0)\}$ with the same spectrum \cite[Theorem VIII.24(a)]{RSI}. Again by iteration, we have
\[
\sp L_0 = \sp(T^\gamma)^{\gamma^{-1}}(1)\subset
\sp L^\gamma_0 = \sp T^\gamma(1) \subset \sp L_0,
\]
therefore, all these sets must coincide. In particular, $h$ is the lowest eigenvalue of $L^\gamma_0$.
\end{proof}

As $T$ and $T^\gamma$ are equivalent as irreducible representations of $\vect(S^1)$
and thus of the Virasoro algebra,
there is a unitary intertwiner $U(\gamma)$, defined up to a scalar such that $U(\gamma)T(f)=T^{\gamma}(f)U(\gamma)$.

\begin{corollary}\label{cr:projective}
The map $\gamma\mapsto U(\gamma)$ where $\gamma\in\D^s(S^1)$, $s>3$,
is a unitary projective representation of $\D^s(S^1)$, i.e.\!
$U(\gamma_1 \gamma_2)=U(\gamma_1)U(\gamma_2)$ up to a phase factor.
\end{corollary}
\begin{proof}
We know that for every $\gamma\in\D^s(S^1)$
\begin{align*}
U(\gamma)T(f)&=T^{\gamma}(f)U(\gamma)
\end{align*}
holds for every $f\in\vect(S^1)$. So
\begin{align*}
U(\gamma_1)U(\gamma_2)T(f)&=U(\gamma_1)T^{\gamma_2}(f)U(\gamma_2)=U(\gamma_1)(T(\gamma_{2*}f)-\beta(\gamma_2,f))U(\gamma_2)=\\
&=(T^{\gamma_1}(\gamma_{2*}f)-\beta(\gamma_2,f))U(\gamma_1)U(\gamma_2)=\\
&=(T((\gamma_1 \gamma_2)_*f)-\beta(\gamma_1,\gamma_{2*}f)-\beta(\gamma_2,f))U(\gamma_1)U(\gamma_2).
\end{align*}
Consequently by the computations of  Proposition \ref{lm:composition}
\[
U(\gamma_1)U(\gamma_2)T(f)=T^{\gamma_1 \gamma_2}(f)U(\gamma_1)U(\gamma_2),
\]
therefore $U(\gamma_1 \gamma_2)=U(\gamma_1)U(\gamma_2)$ up to a phase
because we are dealing with irreducible representations of the Virasoro algebra.
\end{proof}

\begin{corollary}\label{cr:continuityB(H)}
Let $U=U_{(c,h)}$ be the irreducible unitary projective representation of $\diff(S^1)$ with central charge $c$ and lowest weight $h$.
Then $U$ extends to a strongly continuous irreducible unitary projective representation of $\D^s(S^1)$, $s>3$.
\end{corollary}
\begin{proof}
The only thing that remains to be proven is continuity, namely that the action $\alpha:\D^s(S^1)\rightarrow
\Aut(\B(\H(c,h)))$, $\gamma\mapsto \Ad U(\gamma)$ is pointwise continuous in the strong operator topology of $\B(\H(c,h))$.

Let $\lbrace\gamma_n\rbrace\subset\diff(S^1)$, $\gamma\in\D^s(S^1)$ with $\gamma_n\rightarrow \gamma$ in the topology of $\D^s(S^1)$. Then
\[
\lim_{n\rightarrow \infty}U(\gamma_n)e^{itT(f)}U(\gamma_n)^{*}=\lim_{n\rightarrow\infty}e^{itT^{\gamma_n}(f)}=e^{itT^{\gamma}(f)}
\]
where the limit is meant in the strong operator topology.
By taking $f=1$, we obtain the convergence of $L_0^{\gamma_n}$ to $L_0^\gamma$
in the strong resolvent sense.
As they are in the $(c,h)$-representation of the Virasoro algebra,
the lowest eigenprojections $E_0, E_0^\gamma$ are one-dimensional,
and it holds that $\lim_{n\to \infty}\Ad U(\gamma_n)(E_0) = E_0^\gamma$.
Let $\Omega, \Omega^\gamma$ be the lowest eigenvectors.
By fixing the scalars, we may assume that $\Omega^{\gamma_n} := U(\gamma_n)\Omega \to \Omega^\gamma$.

With this $U(\gamma_n)$ with fixed phase,
the sequence
\[
 U(\gamma_n)e^{iT(f_1)}\cdots e^{iT(f_k)}\Omega
 = e^{iT^{\gamma_n}(f_1)}\cdots e^{iT^{\gamma_n}(f_k)}\Omega^{\gamma_n}
\]
is convergent to $e^{iT^{\gamma}(f_1)}\cdots e^{iT^{\gamma}(f_k)}\Omega^{\gamma}$,
because all the operators $e^{iT^{\gamma_n}(f_1)},\cdots, e^{iT^{\gamma_n}(f_k)}$
are uniformly bounded and convergent in the strong operator topology.
Since vectors of the form $e^{iT(f_1)}\cdots e^{iT(f_k)}\Omega$ span a dense subspace of the whole Hilbert space $\H(c,h)$,
together with the uniform boundedness of $U(\gamma_n)$,
we obtain the convergence of $U(\gamma_n)$ to $U(\gamma)$ in the strong operator topology.

The claimed continuity follows from this, because for any $x\in \B(\H)$,
$\Ad U(\gamma_n)(x)$ is convergent in the strong operator topology, again
because $U(\gamma_n)$ is uniformly bounded.

\end{proof}

\begin{corollary}\label{cr:diff4}
Let $U=U_{(c,h)}$ be the irreducible unitary projective representation of $\diff(S^1)$ with central charge $c$ and lowest weight $h$.
Then $U$ extends to a strongly continuous irreducible unitary projective representation of $\diff^k(S^1)$ with $k\geq4$.
\end{corollary}
\begin{proof}
 This is an immediate corollary of the continuous embedding $\diff^k(S^1) \hookrightarrow \D^s(S^1)$, $s \le k$.
\end{proof}

\begin{remark}\label{remarkGW} Our argument for the construction of projective representations of $\D^s(S^1)$ can be used to simplify
the proof of the integrability of the irreducible unitary positive energy representations of the Virasoro algebra
to strongly continuous projective unitary representations of $\diff(S^1)$.
Such a proof was first given in \cite[Section 3, Theorem 4.2]{GW85} by realizing them in the oscillator algebra.
One can do it now only within the Virasoro algebra as follows.

Besides the energy-bounds ({\it a priori estimates}) in \cite[Section 2]{GW85}, see also \cite{BS90}, which are used in \cite{CW05} and are crucial to our proof,
we also used \eqref{eq:Tgammasmooth} coming from \cite{GW85}.
More precisely, we used the fact that for every $\gamma \in \diff(S^1)$ there is a unitary operator $U(\gamma)$ such that 
$U(\gamma)T(f)U(\gamma)^* = T^\gamma(f)$ for all $f \in \vect(S^1)$ and $U(\gamma)\dom(L_0) = \dom(L_0)$. This can be proved directly following the strategy in pages 1100-1101 of \cite{CKL08},
see also the proof of \cite[Proposition 6.4]{CKLW18}. One only needs some of the direct consequences of the energy bounds proved in 
\cite[Section 2]{Toledano-Laredo99-1}. We outline the arguments here:
\begin{itemize}
 \item Since $\diff(S^1)$ is simple \cite[Remark 1.7]{Milnor84}, it is generated by exponentials,
 because the subgroup generated by exponentials is a normal subgroup.
 \item By the proof of Corollary \ref{cr:projective}, the set of $\gamma$ such that a unitary $U(\gamma)$ with
 the required properties exists forms a subgroup of $\diff(S^1)$. Hence, it is enough to consider the special case where $\gamma = \Exp(g)$ for 
 $g \in \vect(S^1)$.
 \item It follows from the {\it linear energy-bounds} by \cite[Proposition 2.1]{Toledano-Laredo99-1} that 
$e^{itT(g)}\dom(L_0^k) = \dom(L_0^k)$ for all positive integers $k$ and all $t \in \mathbb{R}$. As a consequence 
$e^{itT(g)} C^\infty(L_0) = C^\infty(L_0)$ for all $t \in \mathbb{R}$.
 \item Now, let $\xi \in C^\infty(L_0)$
and let $\xi(t) = T^{\Exp(tg)}(f)e^{itT(g)}\xi$.  By \cite[Corollary 2.2]{Toledano-Laredo99-1} we have
$\frac{d}{dt} e^{itT(g)}\xi = i e^{itT(g)} T(g)\xi$ in the graph topology of $\dom(L_0^k)$ for all positive integers $k$. It then follows from the energy bounds that $\frac{d}{dt}\xi(t) = i T(g)\xi(t)$. Hence, $\xi(t) = e^{itT(g)}T(f)\xi$ for all 
$\xi \in C^\infty(L_0)$ so that  $T^{\Exp(tg)}(f) =  e^{itT(g)}T(f) e^{-it T(g)}$ which is the required relation.
Continuity of $U$ follows as in Corollary \ref{cr:continuityB(H)}.
\end{itemize}

\end{remark}

\subsection{Direct sum of irreducible representations}
Here we prove that every positive
energy projective unitary representation of $\diff(S^1)$ extends to a unitary projective representation of $\D^s(S^1)$ for $s>3$. A similar result holds for the universal covering groups
provided that the representation is assumed to be a direct sum of
irreducibles.
This is not an immediate consequence of Corollary \ref{cr:continuityB(H)},
because, in general, the direct sum of projective representations does not make sense:
$\U(\H_j)/\CC$ is not a linear space.
On the other hand, if we have \textit{multiplier representations} of a group $G$ with the same cocycle,
$U_j(g_1)U_j(g_2) = \sigma(g_1,g_2)U_j(g_1 g_2)$ where $\sigma(g_1,g_2)$ is a 2-cocycle $H^2(G,\CC)$
of G, then the direct sum $\bigoplus_j U_j(g)$ is again a multiplier representation
with the same cocycle $\sigma$.
If we are interested in a projective representation of a certain quotient $G/H$ by a normal subgroup $H$
we have to make sure that the direct sum $\bigoplus U_j(h)$ reduces to a scalar when $h \in H$.

\paragraph{Continuous fragmentation of $\widetilde{\D^s(S^1)}$.}
Let $I$ be a proper open interval of $S^1$ and $I^\prime = (S^1\setminus I)^{\circ}$ be the interior of its complement. We denote by $\overline{I}$ the closure of $I$. $\diff(I)$ (resp.\! $\D^s(I)$) denotes the subgroup of diffeomorphisms $\diff(S^1)$ (resp.\! $\D^s(S^1)$) 
such that $\gamma(x)=x$ for $x\in I^\prime$. We also say that  $\gamma\in \diff(I)$ (resp.\! $\gamma\in\D^s(I)$) is supported in $I$.

Let $\{I_j\}_{j=1,2,3}$ be a cover of the unit circle as Fig.\! \ref{fig:intervals}.
Let us name the end points of the intervals: $I_k = (a_k, b_k)$.
We also take slightly smaller intervals $\hat I_k = (\hat a_k, \hat b_k) \subset I_k$
which still provide a cover of $S^1$, and take points $\breve a_1\in (a_1,\hat a_1), \breve b_1\in (\hat b_1, b_1)$,
c.f.\! \cite{DFK04}.
Furthermore, we take $\hat b_2, \check b_2$ such that $\hat a_1 < \hat b_2 < \check b_2 < b_2$.
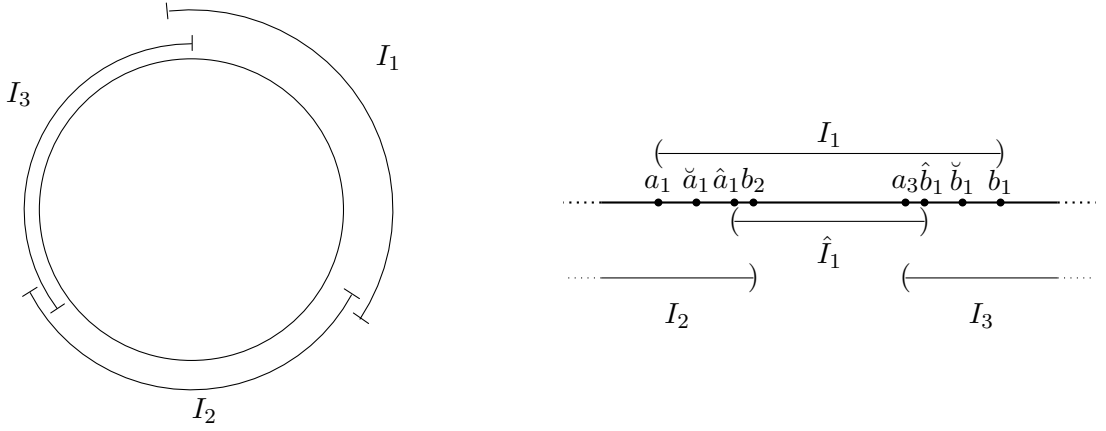
\begin{figure}[ht]
\centering
\begin{tikzpicture}[line cap=round,line join=round,>=triangle 45,x=1.0cm,y=1.0cm]
\clip(-2.43,-0.61) rectangle (6.15,5.35);
\draw(1.4,2.42) circle (2cm);
\draw [shift={(1.4,2.42)}] plot[domain=1.57:3.78,variable=\t]({1*2.2*cos(\t r)+0*2.2*sin(\t r)},{0*2.2*cos(\t r)+1*2.2*sin(\t r)});
\draw [shift={(1.4,2.42)}] plot[domain=3.62:5.79,variable=\t]({1*2.39*cos(\t r)+0*2.39*sin(\t r)},{0*2.39*cos(\t r)+1*2.39*sin(\t r)});
\draw [shift={(1.4,2.42)}] plot[domain=-0.58:1.68,variable=\t]({1*2.65*cos(\t r)+0*2.65*sin(\t r)},{0*2.65*cos(\t r)+1*2.65*sin(\t r)});
\draw (1.07,5.16)-- (1.09,4.95);
\draw (3.53,1.05)-- (3.73,0.91);
\draw (3.41,1.37)-- (3.61,1.24);
\draw (1.41,4.73)-- (1.41,4.53);
\draw (-0.45,1.05)-- (-0.27,1.17);
\draw (-0.82,1.27)-- (-0.63,1.39);
\draw (3.69,4.73) node[anchor=north west] {$I_1$};
\draw (1.27,0.05) node[anchor=north west] {$I_2$};
\draw (-1.17,4.25) node[anchor=north west] {$I_3$};
\end{tikzpicture}
\begin{tikzpicture}[path fading=north,scale=0.5]
         \draw [thick] (-6,0) --(6,0);
         \draw [thick,dotted] (-6,0) --(-7,0);
         \draw [thick,dotted] (6,0) --(7,0);
         \node at(0,1.8) {$I_1$};
         \draw [] (-2.5,-0.5) node{$($}--(2.5,-0.5)node{$)$};
         \node at(0,-1.3) {$\hat I_1$};
         \draw [] (-4.5,1.3) node{$($}--(4.5,1.3)node{$)$};
         \node at(-4.5,0.5) {$a_1$};
         \fill (-4.5,0) circle[radius=3pt];
         \node at(4.5,0.55) {$b_1$};
         \fill (4.5,0) circle[radius=3pt];
         \node at(-3.5,0.6) {$\breve a_1$};
         \fill (-3.5,0) circle[radius=3pt];
         \node at(-2.7,0.6) {$\hat a_1$};
         \fill (-2.5,0) circle[radius=3pt];
         \node at(2.7,0.65) {$\hat b_1$};
         \fill (2.5,0) circle[radius=3pt];
         \node at(3.5,0.65) {$\breve b_1$};
         \fill (3.5,0) circle[radius=3pt];
         \draw [dotted] (-6,-2) --(-7,-2);
         \node at(-4,-3) {$I_2$};
         \node at(-2,0.6) {$b_2$};
         \fill (-2,0) circle[radius=3pt];
         \draw [] (6,-2) --(2,-2)node{$($};
         \draw [] (-6,-2) --(-2,-2)node{$)$};
         \draw [dotted] (6,-2) --(7,-2);
         \node at(4,-3) {$I_3$};
         \node at(2,0.5) {$a_3$};
         \fill (2,0) circle[radius=3pt];
         \node at(0,-6) {};
\end{tikzpicture}
\caption{The covering of the unit circle.}
 \label{fig:intervals}
\end{figure}

Any given diffeomorphism $\g$ can be written as a product of
elements supported in $I_k$. This is known as
fragmentation (see \cite{Mann15} and references therein).
We need a slightly refined version of it, namely, if $\g$ is in a small neighborhood
$\V$ of the unit element $\i$, then we can take the fragments $\g_k$ also in a small,
but larger neighborhood $\hat \V$. The precise statement is the following.

\begin{lemma}\label{lm:fragmentation}
Let $s>\frac32$ and $k\in\{1,2,3\}$. There is a neighborhood $\V$ of the unit element $\iota$ of $\widetilde{\D^s(S^1)}$ and continuous localizing maps $\chi_k: \V \to$ $\widetilde{\D^s(I_k)}$ with
 \[
  \g = \chi_1(\g)\chi_2(\g)\chi_3(\g)
 \]
 and $\chi_k(\iota) = \iota$, $\supp \chi_k(\g) \subset I_k$, where
 $\supp \g := \overline{\{\theta \in S^1: \g(\theta) \neq \theta\}}$.
 If $\supp \g \subset \breve I_k\cup \breve I_{k+1}$, then $\chi_{k+2}(\g) = \iota$, where the indices $k+1$ and $k+2$ are considered $\mod 3$ as elements of $\{1,2,3\}$. 
 \end{lemma}
\begin{proof}
We may assume without loss of generality that
$0 < a_1 < \breve a_1 < \hat a_1 < b_2 < a_3 < \hat b_1 < \breve b_1 < b_1 < 2\pi$,
(see Figure \ref{fig:intervals}).

Let us take a smooth $2\pi$-periodic function $D_{\mathrm{c},1}$ with $D_{\mathrm{c},1}(t)=1$ for $t\in \hat I_1 = [\hat a_1, \hat b_1]$
and $D_{\mathrm{c},1}(t)=0$ for $t\in [0,\breve a_1]\cup [\breve b_1,2\pi]$
and $0 \le D_{\mathrm{c},1}(t) \le 1$ everywhere.
Let $0 \le D_{\mathrm{l},1}(t) \le 1$ be another smooth $2\pi$-periodic function with support in $(a_1,\breve a_1)$
and with $\int_0^{2\pi}D_{\mathrm{l},1}(t)dt = \int_{a_1}^{\breve a_1}D_{\mathrm{l},1}(t)dt= \frac12(\breve a_1-a_1)$
(which is possible because the interval $(a_1,\breve a_1)$ is longer than $\frac12(a_1,\breve a_1)$).
Similarly, let $0 \le D_{\mathrm{r},1}(t) \le 1$ be a smooth $2\pi$-periodic function
with support in $(\breve b_1,b_1)$ and with $\int_0^{2\pi}D_{\mathrm{r},1}(t)dt=\frac12(b_1 - \breve b_1)$.

We consider the following neighborhood of the unit element of $\widetilde{\D^s(S^1)}$
\[
 \V_{\e}\coloneqq \left\{\gamma \in \widetilde{\D^s(S^1)}: |\gamma(\theta)-\i(\theta)|<\e, |\gamma^{\prime}(\theta)-1|<\e
 \;\text{ for }\theta\in[0,2\pi]\right\}.
\]
Note that since $s>3/2$, $\V_\e$ is open by the Sobolev-Morrey embedding theorem.

Suppose $\gamma \in \V_{\e}$.
We set
\begin{align*}
M &:= \max\left\{D_{\mathrm{c}, 1}(t),  t \in[0,2\pi]\right\}
\end{align*}
and define the constant $\a_1(\gamma)$ by
\begin{align}\label{eq:alpha}
 \alpha_1(\gamma) = \frac2{\breve a_1 - a_1}\left(\gamma(\hat a_1)-\hat a_1 - \int_0^{\hat a_1} (\gamma^{\prime}(t)-1)D_{\mathrm{c},1}(t)dt\right).
\end{align}
It follows that 
\begin{equation}\label{estalpha}
|\alpha_1(\gamma)|\leq \frac {2}{|\breve a_1 - a_1|} \e (1+\hat{a}_1M)
\end{equation} by the definition of $\V_{\e}$ and

\[
 \gamma(\hat a_1)=\int_0^{\hat a_1} ((\gamma^{\prime}(t)-1)D_{\mathrm{c},1}(t)+1+\alpha_1(\gamma)D_{\mathrm{l},1}(t))dt. 
\]
Similarly, set the constant $\beta_1(\gamma)$ by
\begin{align}\label{eq:beta}
 \b_1(\gamma) &\;\;= \frac{-2}{b_1 - \breve b_1}\left(\int_0^{2\pi} ((\gamma^{\prime}(t)-1)D_{\mathrm{c},1} (t)+\alpha_1(\gamma)D_{\mathrm{l},1}(t))dt\right) \\
 &\left( =\frac2{b_1 - \breve b_1}\left(\hat b_1 - \g(\hat b_1) - \int_{\hat b_1}^{b_1} (\gamma^{\prime}(t)-1)D_{\mathrm{c},1} (t)\right) \right), \nonumber
\end{align}
then it follows that 
\begin{align}\label{estbeta}
|\beta_1(\gamma)|\leq \frac {2}{|b_1 - \breve b_1|} \e (|\hat{b}_1-b_1| M+1)
\end{align}
 and
\begin{align*}
 b_1 = \int_0^{b_1} ((\gamma^{\prime}(t)-1)D_{\mathrm{c},1} (t)+1+\alpha_1(\gamma)D_{\mathrm{l},1}(t)+\beta_1(\gamma)D_{\mathrm{r},1}(t))dt.
\end{align*}
Now, the function
\begin{align}\label{eq:gamma1}
 \gamma_1(\theta)=\int_0^\theta((\gamma^\prime (t)-1)D_{\mathrm{c},1}(t)+1+\alpha_1(\gamma)D_{\mathrm{l},1}(t)+\beta_1(\gamma)D_{\mathrm{r},1}(t))dt
\end{align}
is $2\pi$-periodic, the first derivative
\begin{align*}
 \gamma'_1(\theta)= (\gamma^\prime (\theta)-1)D_{\mathrm{c},1}(\theta)+1+\alpha_1(\gamma)D_{\mathrm{l},1}(\theta)+\beta_1(\gamma)D_{\mathrm{r},1}(\theta)
\end{align*}
is positive by \eqref{estalpha}, \eqref{estbeta} if $\e$ is taken sufficiently small and $\gamma_1^{\prime}-1\in H^{s-1}(S^1)$ (by Lemma \ref{lm:sobolevalgebra}, using that $\g-\iota\in H^s(S^1)$),
therefore, $\gamma_1$ can be regarded as an element in $\widetilde{D^s(S^1)}$.
It also has the desired properties,
namely $\gamma_1(\theta)=\theta$ for $\theta\in I_1'$ and $\gamma_1(\theta)=\gamma(\theta)$ for $\theta\in \hat I_1$.
Note that the assignment $\V_{\e}\rightarrow \widetilde{\D^s(S^1)}$, $\gamma\rightarrow\g_1$ is continuous
by \eqref{eq:gamma1}\eqref{eq:alpha}\eqref{eq:beta} and Lemma \ref{lm:derivative}.

We choose $\e$ such that $\gamma_1^{\prime}$ is positive for $\gamma\in\V_\e$. 
Now the assignment $\V_{\e}\rightarrow \widetilde{\D^s(S^1)}$, $\gamma\rightarrow\g\g_1^{-1} $ is continuous
by Lemma \ref{lm:sobolevgroup}.
We take $\V \subset \V_\e$ to be the neighborhood of the identity of $\widetilde{D^s(S^1)}$
such that for $\gamma\in\V$ we have $\g\g_1^{-1}\in\V_{\e_1}$ where $\e_1$ is small enough that
we obtain $\gamma_2\in\widetilde{D^s(S^1)}$ (in particular $\gamma_2^{\prime}$ is positive)
if we do an analogous construction on $I_2$ for $\gamma \gamma_1^{-1}$.

For $\gamma\in\V$ we set $\chi_1(\g) = \g_1$.
The continuity of the map $\chi_1$ in the topology of $\widetilde{\D^s(S^1)}$ is clear from \eqref{eq:gamma1} and \eqref{eq:alpha}\eqref{eq:beta}.

Next we construct $\chi_2(\g)$.
By construction $(\g\g_1^{-1})(\theta) = \theta$ for $\theta \in \hat I_1$,
therefore , $\supp \g\g_1^{-1} \subset I_2 \cup I_3$.
We can apply an analogous construction to $I_2$ and
$\gamma\g_1^{-1}$ to obtain $\g_2$ such that
$\supp \g_2 \subset \hat I_2, \g_2(\theta) = (\g\g_1^{-1})(\theta)$ for $\theta \in \hat I_2$.
In this way we obtain the continuous map $\chi_2(\g) := \g_2$.
Furthermore, by our choice $\hat a_1 < \hat b_2 < \check b_2 < b_2$, $\g_2(\theta) = (\g\g_1^{-1})(\theta)$
for $\theta \in \hat I_1$ where both are equal to $\theta$,
hence for $\hat I_1 \cup \hat I_2$.

Now we have $(\g\g_1^{-1}\g_2^{-1})(\theta) = \theta$ for $\theta \in \hat I_1 \cup \hat I_2$,
and as $\{\hat I_k\}$ is a cover of $S^1$, $(\hat I_1 \cup \hat I_2)' \subset \hat I_3$.
Therefore, if we set $\chi_3(\g) = \g\g_1^{-1}\g_2^{-1}$, it is supported in $\hat I_3 \subset I_3$
and the map $\chi_3$ is continuous because it is a composition of continuous maps (Lemma \ref{lm:sobolevgroup}).
\end{proof}

If $\g$ is already localized, we can have the following improvement.
\begin{lemma}\label{lm:fragmentation-local}
Let $s>\frac{3}{2}$, $k \in \{1,2,3\}$ and $\tilde I_k = I_k \cup I_{k+1}$ where the index $k+1$ is considered $\mod 3$ as an element of $\{1,2,3\}$.
Then there exists a neighborhood $\V$ of the unit element $\iota$ of $\widetilde{\D^s(S^1)}$
and continuous localizing maps
\begin{align*}
 \chi^{(k)}_k&: \V\cap \widetilde{\D^s(\tilde I_k)} \to \widetilde{\D^s(I_k)}, \\
\chi^{(k)}_{k+1}&: \V\cap \widetilde{\D^s(\tilde I_k)} \to \widetilde{\D^s(I_{k+1})},
\end{align*}
such that $\g = \chi^{(k)}_k(\g)\chi^{(k)}_{k+1}(\g)$
 and $\chi^{(k)}_k(\iota) = \chi^{(k)}_{k+1}(\iota) = \iota$.
\end{lemma}
\begin{proof}
 Without loss of generality, we may assume $k=2$.
 This is done by applying the steps of the construction of $\chi_2$ and $\chi_3$ in the proof of Lemma \ref{lm:fragmentation}
 to slightly enlarged $I_2$ and $\hat I_2$, so that $\chi^{(2)}_2(\g)(\theta) = \g(\theta)$  for $\theta \in I_3'$.
\end{proof}

\begin{lemma}\label{lm:localequivalence}
Let $U_{(c,h_1)}, U_{(c,h_2)}$ be irreducible, projective representations of
$\widetilde{\D^s(S^1)}$ with central charge $c$ and lowest weight $h_1,h_2$ respectively,
constructed as in Section \ref{extension}.
Let $I$ be a proper interval of $S^1$. Then the projective representations $U_{(c,h_1)}$ and $U_{(c,h_2)}$
restricted to $\D^s(I)$
are unitarily equivalent. Furthermore, a unitary $U$ intertwines $U_{(c,h_1)}$ and $U_{(c,h_2)}$ restricted to $\D^s(I)$ if and only if it intertwines $T_{(c,h_1)}(f)$ and $T_{(c,h_2)}(f)$ for every $f\in\vect(S^1)$ with support in $I$.
\end{lemma}
\begin{proof}
Let $\tilde I$ an open proper interval of $S^1$ such that $\tilde I\supset\overline{I}$.
By \cite[Theorem 5.6]{Weiner17} there exists a unitary $W$ which intertwines the representations $U_{(c,h_1)}, U_{(c,h_2)}$
when restricted to $\diff(\tilde I)$. Let $\g\in\D^s(I)$, then by Lemma \ref{lem:localapprox} there exists a sequence of $C^{\infty}$-diffeomorphisms
$\lbrace\g_n\rbrace\subset\diff(\tilde I)$ converging to $\g$. By Corollary \ref{cr:continuityB(H)},
\begin{align*}
\Ad WU_{(c,h_1)}(\g)W^*& 
= \Ad \lim_{n\rightarrow\infty} WU_{(c,h_1)}(\g_n)W^*
= \Ad \lim_{n\rightarrow\infty} U_{(c,h_2)}(\g_n) 
= \Ad U_{(c,h_2)}(\g).
\end{align*}
The last assertion follows from \cite[Lemma 2.1]{Weiner17}.
\end{proof}

We are going to show that we can take the direct sum of irreducible projective representations of $\D^s(S^1)$,
$\{U_{(c,h_j)}\}$, with the same central charge $c$ but possibly different lowest weights $\{h_j\}$,
where differences $h_j - h_{j'}$ are integers.
We split the proof into two steps.
First, we make $U_{(c,h_j)}$ into continuous multiplier representations with the same cocycle in some neighborhood $\V$
of the identity diffeomorphism $\iota\in \widetilde{\D^s(S^1)}$.
Then it is straightforward to take the direct sum.
Next, we show that the direct sum representation reduces to a projective representation of $\D^s(S^1)$
if the differences $h_j - h_{j'}$ are integers.

Let $G$ and $G'$ be two topological groups. Given a neighborhood $\V$ of the identity in $G$, a continuous map
$\mu:\V\rightarrow G'$ is a local homomorphism if $\mu(g_1)\mu(g_2)=\mu(g_1g_2)$ for all $g_1,g_2\in\V$ and $g_1g_2\in\V$.

We say that a map $U$ is a local unitary multiplier representation of a topological group $G$ on a neighborhood $\V$ of the identity  if $U$ is a map from $\V$ to the unitary group $\U(\H)$ of a Hilbert space $\H$ which satisfies the equality $U(g_1)U(g_2)=\sigma(g_1,g_2)U(g_1g_2)$, where $\sigma:\V\times\V\rightarrow\mathbb{T}$ and $\sigma(g_1,g_2)\sigma(g_1g_2,g_3)=\sigma(g_1,g_2g_3)\sigma(g_2,g_3)$ whenever $g_1,g_2,g_3$, $g_1g_2$ and $g_2g_3$ are in $\V$.
The following is obtained by reversing the idea of \cite{Tanimoto18-2}.

\begin{proposition}\label{pr:directsum}
 Let $s>3$. For a family $\{(c, h_j)\}$ of pairs with the same central charge $c$,
 there is a neighborhood $\V$ of $\widetilde{\D^s(S^1)}$ such that the irreducible unitary projective representations $U_{(c,h_j)}$ lift to local multiplier representations of $\V$ with the same cocycle $\sigma_c(\cdot,\cdot)$.
\end{proposition}
\begin{proof}
Let us take $h_1$.
As explained in Section \ref{projective} (cf. also \cite[Proposition 12.44]{Moretti17}), in a neighborhood $\hat{\V}$ of the identity $\iota\in\widetilde{\D^s(S^1)}$,
$U_{(c,h_1)}$ lifts to a continuous multiplier representation, with some continuous cocycle $\sigma_c(\cdot,\cdot)$,
which we will denote by $U_1$.

Because $\widetilde{\D^s(S^1)}$ is a topological group, and by Lemmas \ref{lm:fragmentation}, \ref{lm:fragmentation-local},
for each neighborhood $\W$, there is a smaller neighborhood $p(\W)$ such that
$p(\W)^2 \subset \W$ and $\chi_k(\g), \chi^{(k)}_k(\g), \chi^{(k)}_{k+1}(\g) \subset \W$ for $\g \in p(\W)$.
We take $\V = p^{11}(\hat \V) = \underset{11\text{-times}}{\underbrace{p(p(p(\cdots \hat \V\cdots)))}}$.

\paragraph{Construction of multiplier representations $U_j$.}
We show that we can take $U_j$ with the same cocycle $\sigma_c(\cdot,\cdot)$. 
Let us take a local multiplier representation $U_1 = U_{(c,h_1)}$ with $(c,h_1)$.

We fix a covering $\{I_k\}$ of $S^1$ as in Lemma \ref{lm:fragmentation}.
For $\g \in p(\hat \V)$, we define $U_j$ as follows:
By Lemma \ref{lm:localequivalence},
there are unitary intertwiners $\{V_{j,k}\}$ between $U_{(c,h_1)}$ and $U_{(c,h_j)}$ restricted to $\D^s(I_k)$.
We set
\[
 U_j(\chi_k(\gamma))=\Ad V_{j,k}(U_1(\chi_k(\gamma))),
\]
which makes sense because $p(\hat \V) \subset \hat \V$.
Note that $U_j(\chi_k(\g))$ does not depend on the choice of unitary intertwiner $V_{j,k}$,
since, if $V_{j,k}$ and $\hat{V}_{j,k}$ are both unitary intertwiners,
then by Lemma \ref{lm:localequivalence}
\[
 \Ad V_{j,k}^*\hat{V}_{j,k}(U_j(\chi_k(\g)))=U_j(\chi_k(\g)) 
\]
for $\g$ smooth, and by continuity of $U_1$ for $\chi_k(\g)\in \D^s(I_k) \cap \hat \V$.

Let us denote $\g_k = \chi_k(\g)$ for simplicity.
Now, since $\g=\g_1\g_2\g_3$ with $\g_k\in \D^s(I_k)\cap \hat \V$, we can define $U_j(\g)$ by
\begin{align}\label{eq:defpi}
 U_j(\g)=U_j(\g_1)U_j(\g_2)U_j(\g_3)\sigma_c(\g_1,\g_2)^{-1}\sigma_c(\g_1\g_2,\g_3)^{-1},
\end{align}
and note that the corresponding equation holds for $U_1$.
\paragraph{Well-definedness.}
We used a particular set of maps $\chi_k$ to define $U_j$, 
but actually
they do not depend on the choice of such map $\chi_k$ if $\g$ satisfies certain properties and is sufficiently close to $\i$.
Namely, we take two decompositions $\g = \g_1\g_2\g_3 = \g'_1\g'_2\g'_3$
where $\g_k, \g_k' \in \D^s(I_k) \cap p^5(\hat\V)$.

It holds that $\g_3^{-1}\g_2^{-1}\g_1^{-1}\g'_1\g'_2\g'_3 = \iota$ in $\widetilde{\D^s(S^1)}$
and $U_1(\g_1)^* = \sigma_c(\g_1,\g_1^{-1})U_1(\g_1^{-1})$,
hence we have
\[
 \sigma_c(\g_1,\g_2,\g_3,\g'_1,\g'_2,\g'_3) := U_1(\g_3)^*U_1(\g_2)^*U_1(\g_1^{-1}\g'_1)U_1(\g'_2)U_1(\g'_3) \in \mathbb{C}.
\]
Furthermore, as $U_1$ is a multiplier representation in $\hat \V$, we have
\begin{align*}
 U_1(\g) &= U_1(\g_1)U_1(\g_2)U_1(\g_3)\sigma_c(\g_1,\g_2)^{-1}\sigma_c(\g_1\g_2,\g_3)^{-1} \\
 &= U_1(\g'_1)U_1(\g'_2,)U_1(\g'_3)\sigma_c(\g'_1,\g'_2)^{-1}\sigma_c(\g'_1\g_2',\g'_3)^{-1}.
\end{align*}
By putting all factors in one side, we obtain
\begin{align}\label{eq:c6}
  \sigma_c(\g_1,\g_2,\g_3,\g'_1,\g'_2,\g'_3)\sigma_c(\g_1^{-1},\g'_1)\sigma_c(\g_1,\g_1^{-1})\sigma_c(\g_1,\g_2)\sigma_c(\g_1\g_2,\g_3)\sigma_c(\g'_1,\g'_2)^{-1}\sigma_c(\g'_1\g_2',\g'_3)^{-1} = 1.
\end{align}

Note that $U_j$ is unitarily equivalent to $U_1$ on any proper interval, therefore,
$U_j(\g_1)^*U_j(\g'_1) = \sigma_c(\g_1^{-1},\g'_1)\sigma_c(\g_1,\g_1^{-1})U_j(\g_1^{-1}\g'_1)$,
and $\g_1^{-1}\g'_1 = \g_2\g_3\g_3^{\prime-1}\g_2^{\prime-1}$ has support in $I_2\cup I_3$.
Then we can again use the unitary equivalence between $U_j$ and $U_1$ on $I_2\cup I_3$
to obtain 
\[
 U_j(\g_3)^*U_j(\g_2)^*U_j(\g_1^{-1}\g'_1)U_j(\g'_2)U_j(\g'_3) = \sigma_c(\g_1,\g_2,\g_3,\g'_1,\g'_2,\g'_3),
\]
which is, by \eqref{eq:c6}, equivalent to the equality
\begin{align*}
 &U_j(\g_1)U_j(\g_2)U_j(\g_3)\sigma_c(\g_1,\g_2)^{-1}\sigma_c(\g_1\g_2,\g_3)^{-1} \\
 =\;& U_j(\g'_1)U_j(\g'_2)U_j(\g'_3)\sigma_c(\g'_1,\g'_2)^{-1}\sigma_c(\g'_1\g_2',\g'_3)^{-1}.
\end{align*}
In other words, $U_j$ is well-defined on $p^6(\hat \V)$.

\paragraph{Cocycle relations.}
Next we show that $U_j$ is a local multiplier representation on $\V$.
Let $\g,\g'\in \V = p^{11}(\hat \V)$ and we take decompositions $\g=\g_1\g_2\g_3, \g'=\g'_1\g'_2\g'_3$.
We first look at the product $\g_3\g'_1$. This is supported in $I_1\cup I_3$, and
we can find another decomposition $\g_3\g'_1 = \g''_1\g''_3$ using Lemma \ref{lm:fragmentation-local}, where $\g''_j \in \D^s(I_j) \cap p^{8}(\hat \V)$.
By repeating such operations and taking new decompositions in proper intervals, we find
\begin{align*}
 \g\g' &= \g_1\g_2\g_3\g'_1\g'_2\g'_3 \\
 &= \g_1\g_2\g''_1\g''_3\g'_2\g'_3 \\
 &= \g_1\g'''_1\g'''_2\g''''_2\g''''_3\g'_3,
\end{align*}
where $\g_j^{(k)} \in \D^s(I_j) \cap p^6(\hat \V)$.

Again, by considering the multiplier representation $U_1$, we can prove the following relations
\begin{align}\label{eq:defc4}
 \begin{array}{rl}
 U_1(\g_3)U_1(\g'_1) &= U_1(\g''_1)U_1(\g''_3)\sigma_c(\g_3,\g'_1,\g''_1,\g''_3),\\
 U_1(\g_2)U_1(\g''_1) &= U_1(\g'''_1)U_1(\g'''_2)\sigma_c(\g_2,\g''_1,\g'''_1,\g'''_2),\\
 U_1(\g''_3)U_1(\g'_2) &= U_1(\g''''_2)U_1(\g''''_3)\sigma_c(\g''_3,\g'_2,\g''''_2,\g''''_3),
 \end{array}
\end{align}
where $\sigma_c(\g_3,\g'_1,\g''_1,\g''_3),\sigma_c(\g_2,\g''_1,\g'''_1,\g'''_2),\sigma_c(\g''_3,\g'_2,\g''''_2,\g''''_3)\in\mathbb{C}$
are defined by these equalities. Therefore, as $U_1$ has the cocycle $\sigma_c$,
\begin{align*}
 &\sigma_c(\g,\g') U_1(\g\g')\\
 &\,=U_1(\g)U_1(\g')\\
 & \begin{array}{l}
  = \sigma_c(\g_1,\g_2)^{-1}\sigma_c(\g_1\g_2,\g_3)^{-1}\sigma_c(\g'_1,\g'_2)^{-1}\sigma_c(\g'_1\g'_2,\g'_3)^{-1}\\
  \quad\times U_1(\g_1)U_1(\g_2)U_1(\g_3)U_1(\g'_1)U_1(\g'_2)U_1(\g'_3)
   \end{array} &\text{ by } \eqref{eq:defpi}\\
 & \begin{array}{l}
 =\sigma_c(\g_1,\g_2)^{-1}\sigma_c(\g_1\g_2,\g_3)^{-1}\sigma_c(\g'_1,\g'_2)^{-1}\sigma_c(\g'_1\g'_2,\g'_3)^{-1}\\
 \quad\times U_1(\g_1)U_1(\g'''_1)U_1(\g'''_2)U_1(\g''''_2)U_1(\g''''_3)U_1(\g'_3)\\
 \quad\times \sigma_c(\g_3,\g'_1,\g''_1,\g''_3)\sigma_c(\g_2,\g''_1,\g'''_1,\g'''_2) \sigma_c(\g''_3,\g'_2,\g''''_2,\g''''_3)
   \end{array} & \text{ by } \eqref{eq:defc4} \\
 & \begin{array}{l}
  =\sigma_c(\g_1,\g_2)^{-1}\sigma_c(\g_1\g_2,\g_3)^{-1}\sigma_c(\g'_1,\g'_2)^{-1}\sigma_c(\g'_1\g'_2,\g'_3)^{-1}\\
 \quad\times \sigma_c(\g_3,\g'_1,\g''_1,\g''_3)\sigma_c(\g_2,\g''_1,\g'''_1,\g'''_2) \sigma_c(\g''_3,\g'_2,\g''''_2,\g''''_3)\\
 \quad\times \sigma_c(\g_1,\g'''_1)\sigma_c(\g'''_2\g''''_2)\sigma_c(\g''''_3\g'_3)\cdot U_1(\g_1\g'''_1)U_1(\g'''_2\g''''_2)U_1(\g''''_3\g'_3)   
   \end{array} \\
 & \begin{array}{l}
   =\sigma_c(\g_1,\g_2)^{-1}\sigma_c(\g_1\g_2,\g_3)^{-1}\sigma_c(\g'_1,\g'_2)^{-1}\sigma_c(\g'_1\g'_2,\g'_3)^{-1}\\
   \quad \times \sigma_c(\g_3,\g'_1,\g''_1,\g''_3)\sigma_c(\g_2,\g''_1,\g'''_1,\g'''_2) \sigma_c(\g''_3,\g'_2,\g''''_2,\g''''_3)\\
   \quad \times \sigma_c(\g_1,\g'''_1)\sigma_c(\g'''_2\g''''_2)\sigma_c(\g''''_3\g'_3)\cdot \sigma_c(\g_1\g'''_1,\g'''_2\g''''_2)\sigma_c(\g_1\g'''_1\g'''_2\g''''_2,\g''''_3\g'_3)U_1(\g\g') 
   \end{array} 
 \end{align*}
or equivalently, the following relation between scalars:
\begin{align}\label{eq:c246}
 \sigma_c(\g,\g') =&\;\sigma_c(\g_1,\g_2)^{-1}\sigma_c(\g_1\g_2,\g_3)^{-1}\sigma_c(\g'_1,\g'_2)^{-1}\sigma_c(\g'_1\g'_2,\g'_3)^{-1} \nonumber \\
 &\times \sigma_c(\g_3,\g'_1,\g''_1,\g''_3)\sigma_c(\g_2,\g''_1,\g'''_1,\g'''_2) \sigma_c(\g''_3,\g'_2,\g''''_2,\g''''_3)\\
 &\times \sigma_c(\g_1,\g'''_1)\sigma_c(\g'''_2\g''''_2)\sigma_c(\g''''_3\g'_3)\cdot \sigma_c(\g_1\g'''_1,\g'''_2\g''''_2)\sigma_c(\g_1\g'''_1\g'''_2\g''''_2,\g''''_3\g'_3). \nonumber
 \end{align}

Since $U_j$ is locally equivalent to $U_1$, the following also follows from \eqref{eq:defc4}:
\begin{align}\label{eq:c4-j}
 \begin{array}{rl}
 U_j(\g_3)U_j(\g'_1) &= U_j(\g''_1)U_j(\g''_3)\sigma_c(\g_3,\g'_1,\g''_1,\g''_3),\\
 U_j(\g_2)U_j(\g''_1) &= U_j(\g'''_1)U_j(\g'''_2)\sigma_c(\g_2,\g''_1,\g'''_1,\g'''_2),\\
 U_j(\g''_3)U_j(\g'_2) &= U_j(\g''''_2)U_j(\g''''_3)\sigma_c(\g''_3,\g'_2,\g''''_2,\g''''_3).
 \end{array}
\end{align}

Now, in order to show that $U_j$ is a local multipler representation with the cocycle $\sigma_c$, we only have to compute
\begin{align*}
 &U_j(\g)U_j(\g') \\
 & \begin{array}{l}
  = \sigma_c(\g_1,\g_2)^{-1}\sigma_c(\g_1\g_2,\g_3)^{-1}\sigma_c(\g'_1,\g'_2)^{-1}\sigma_c(\g'_1\g'_2,\g'_3)^{-1}\\
  \quad \times U_j(\g_1)U_j(\g_2)U_j(\g_3)U_j(\g'_1)U_j(\g'_2)U_j(\g'_3)
  \end{array} &\text{ by } \eqref{eq:defpi} \\
 & \begin{array}{l}
   = \sigma_c(\g_1,\g_2)^{-1}\sigma_c(\g_1\g_2,\g_3)^{-1}\sigma_c(\g'_1,\g'_2)^{-1}\sigma_c(\g'_1\g'_2,\g'_3)^{-1}\\
   \quad \times U_j(\g_1)U_j(\g'''_1)U_j(\g'''_2)U_j(\g''''_2)U_j(\g''''_3)U_j(\g'_3)\\
   \quad \times \sigma_c(\g_3,\g'_1,\g''_1,\g''_3)\sigma_c(\g_2,\g''_1,\g'''_1,\g'''_2) \sigma_c(\g''_3,\g'_2,\g''''_2,\g''''_3)
   \end{array} & \text{ by } \eqref{eq:c4-j} \\
 & \begin{array}{l}
   = \left(\sigma_c(\g_1,\g'''_1)\sigma_c(\g'''_2\g''''_2)\sigma_c(\g''''_3\g'_3) \sigma_c(\g_1\g'''_1,\g'''_2\g''''_2)\sigma_c(\g_1\g'''_1\g'''_2\g''''_2,\g''''_3\g'_3)\right)^{-1} \\
   \quad \times \sigma_c(\g,\g') U_j(\g_1)U_j(\g'''_1)U_j(\g'''_2)U_j(\g''''_2)U_j(\g''''_3)U_j(\g'_3)
   \end{array} & \text{ by } \eqref{eq:c246} \\
 & \begin{array}{l}
  = \left( \sigma_c(\g_1\g'''_1,\g'''_2\g''''_2) \sigma_c(\g_1\g'''_1\g'''_2\g''''_2,\g''''_3\g'_3)\right)^{-1} \\
   \quad \times \sigma_c(\g,\g') U_j(\g_1\g'''_1)U_j(\g'''_2\g''''_2)U_j(\g''''_3\g'_3)
   \end{array} \\
 &=\,\sigma_c(\g,\g')U_j(\g\g'),
\end{align*}
where we used local equivalence between $U_j$ and $U_1$ in the 2nd and 4th equalities,
and the well-definedness (independence of the partition of a group element
into $\D^s(I_k)\cap p^5(\hat \V)$) in the 5th equality. Namely, $U_j$ has the cocycle $\sigma_c$ on $\V = p^{11}(\hat \V)$.
\end{proof}

\paragraph{Direct sum of multiplier representations.}
Since all the projective representations $U_j$ can be made into the local multiplier representations with the same cocycle $\sigma_c$,
the direct sum $U := \bigoplus_j  U_j$ is again a local multiplier representation of $\widetilde{\D^s(S^1)}$ on $\V$.
By forgetting the phase, we can interpret $U$ as a local projective representation of $\V \subset \widetilde{\D^s(S^1)}$,
or in other words, a continuous local group homomorphism from $\V$ into $\U(\H)/\mathbb{T}$ (see Section \ref{projective}),
where $\H = \bigoplus_j \H(c,h_j)$.
As $\widetilde{\D^s(S^1)}$ is simply connected and locally connected,
$U$ extends to a continuous projective representation of $\widetilde{\D^s(S^1)}$ \cite[Theorem 63]{Pontryagin46}.

\begin{theorem}\label{th:sumdiff}
 Let $s>3$. For a family $\{(c, h_j)\}$ of pairs with the same central charge $c$
 such that $h_j - h_{j'} \in \NN$,
 the above defined direct sum projective representation $U$ of $\widetilde{\D^s(S^1)}$ satisfies $U(R(2\pi)) \in \CC$,
 where $R(\cdot)$ is the lift of rotations to $\widetilde{\D^s(S^1)}$,
 or in other words, $U$ gives a projective representation of $\D^s(S^1)$.
\end{theorem}
\begin{proof}
 
  Let $\tilde{U}_{(c,h_j)}$ the irreducible global multiplier representation of $\widetilde{\diff(S^1)}$ with central charge $c$ and lowest weight $h_j$ associated to the Bott-Virasoro cocycle. As a projective representation,
  we have $U\big\vert_{\widetilde{\diff(S^1)}}=\bigoplus_j \tilde{U}_{(c,h_j)}$:
  this is because, by definition of $U$, they agree on a neighborhood of the identity of $\widetilde{\diff(S^1)}$, and since $\widetilde{\diff(S^1)}$ is simply connected they agree globally. Since $\widetilde{\mathrm{PSL}(2, \RR)}$ is a simply connected and simple Lie group, $U\big\vert_{\widetilde{\mathrm{PSL}(2,\RR)}}$
 extends to a true representation of $\widetilde{\mathrm{PSL}(2, \RR)}$
 by changing $U(\g)$ only by a scalar \cite[Theorem 7.1]{Bargmann54}. The lift to a true representation of $\widetilde{\mathrm{PSL}(2,\RR)}$ is unique, since if $V_1$ and $V_2$ are true representations which give rise to the same projective representation, we have that $V_1(g)=\chi(g)V_2(g)$ for all $g\in \widetilde{\mathrm{PSL}(2,\RR)}$, where $\chi$ is a character. Since $\widetilde{\mathrm{PSL}(2,\RR)}$ is a perfect group, $\chi(g)=1$ for all $g$. By the uniqueness of the lift of $U\big\vert_{\widetilde{\mathrm{PSL}(2,\RR)}}$ to a true representation $V$, we have that $V=\bigoplus_j V_{(c,h_j)}$, where $V_{(c,h_j)}$ is the lift of $\tilde{U}_{(c,h_j)}\big\vert_{\widetilde{\mathrm{PSL}(2, \RR)}}$ to a true representation. As we assumed that $h_j - h_{j'}$ are integers, $V(R(2\pi)) \in \CC$.

\end{proof}

From the previous theorem, it follows that every positive
energy projective unitary representation of $\diff(S^1)$ extends to a unitary projective representation of $\D^s(S^1)$ using the following  well-known fact that we here prove for self-containment.
\begin{proposition}\label{pr:compred}
Let $U$ be a positive energy unitary projective representation of $\diff(S^1)$ on the Hilbert space $\H$. Then $U$ is unitarily equivalent to a direct sum of irreducible positive energy unitary projective representation of $\diff(S^1)$
and extends to $\D^s(S^1)$, $s>3$.
\end{proposition}
\begin{proof}
As in the proof of Theorem \ref{th:sumdiff}, we have that $U\big\rvert_{\mathrm{PSL}(2,\RR)}$ can be lifted to
a true representation of $\widetilde{\mathrm{PSL}(2, \RR)}$. Thus we can take the generator of rotations $L_0$ and,
since $e^{i2\pi L_0}\in\CC \1$ from the fact that $U$ is a projective representation of $\diff(S^1)$,
it follows that $L_0$ is diagonalizable with spectrum Sp$(L_0)\subset\{h_1+\mathbb{N}\}$ with $h_1\in\mathbb{R}$,
$h_1\geq 0$. Let $\H^\fin$ be the dense subspace of $\H$ generated by the eigenvectors of $L_0$.
We can apply \cite[Theorem 3.4]{CKLW18} to conclude that there exists a positive energy unitary representation
$\pi_U$ of $\vir$ on $\H^\fin$.

The representation of $\vir$ on $\H^\fin$ is equivalent to an algebraic orthogonal
direct sum of multiples of irreducible positive energy representations of $\vir$ in the following sense.
Let $V_1$ be the smallest $\pi_U$-invariant subspace of $\H^\fin$ which contains $\ker(L_0-h_1\1_{\H^\fin})$
where $h_1$ is the smallest eigenvalue of $L_0$. By induction let $V_n$ be the smallest $\pi_U$-invariant subspace
of $\left(V_1\oplus V_2\oplus\cdots\oplus V_{n-1}\right)^{\perp}\cap\H^\fin$
which contains $\left(V_1\oplus V_2\oplus\cdots\oplus V_{n-1}\right)^{\perp}\cap\ker(L_0-h_n\1_{\H^\fin})$
where $h_n$ is the smallest eigenvalue of $L_0$ restricted to $\left(V_1\oplus V_2\oplus\cdots\oplus V_{n-1}\right)^{\perp}\cap\H^\fin$.
It is straightforward to see that $\H^\fin=\bigoplus_n V_n$ in the algebraic sense.
Now choose an orthonormal basis $\lbrace e^n_j\rbrace$ of $W_n\coloneqq V_n\cap\ker(L_0-h_n\1_{\H^\fin})$.
We define $H_j^n$ to be the smallest $\pi_U$-invariant subspace of $W_n$ which contains the vector $e^n_j$.
By construction $H_j^n$ has no proper $\pi_U$-invariant subspaces, $H_j^n$ and $H_k^n$ are orthogonal subspaces for $j\ne k$
and $\overline{V_n}=\bigoplus_j\overline{H^n_j}$. Let $T$ be the stress-energy tensor associated to the representation $\pi_U$ of $\vir$.
By construction $T(f)|_{H_j^n}$ is essentially self-adjoint on $H^n_j$.

To conclude the decomposition of $U$, we have to show that $e^{iT(f)}\overline{H^n_j}\subset\overline{H^n_j}$ for all $f\in\vect(S^1)$.
We note that $\dom\left(\left(\overline{(T(f)|_{H_j^n})}\right)^\ell\right)\subset \dom(T(f)^\ell)$ and
if $\xi\in\dom\left(\left(\overline{(T(f)|_{H_j^n})}\right)^\ell\right)$ then $\left(\overline{T(f)|_{H_j^n}}\right)^\ell\xi=(T(f))^\ell\xi$.
Thus the analytic vectors for $\overline{(T(f)|_{H_j^n})}$ are also analytic for $T(f)$ and  $e^{i\overline{(T(f)|_{H_j^n})}}\xi=e^{iT(f)}\xi$.
Using the density of the analytic vectors in $\overline{H_j^n}$, we obtain 
that $e^{i\overline{(T(f)|_{H_j^n})}}=e^{iT(f)}\big\vert_{H^n_j}$.
Irreducibility of $U\vert_{\overline{H^n_j}}$ follows because $T\vert_{H^n_j}$ is irreducible.

The extension to $\D^s(S^1)$ is now a mere corollary of Theorem \ref{th:sumdiff}.
\end{proof}

\begin{corollary}\label{cr:diff4red}
Let $U$ be a positive energy unitary projective representation of $\diff(S^1)$ on the Hilbert space $\H$.
Then $U$ is unitarily equivalent to a direct sum of irreducible positive energy unitary projective representation of $\diff(S^1)$
and extends to $\diff^k(S^1)$ with $k\geq 4$.
\end{corollary}
\begin{proof}
 This again follows from Proposition \ref{pr:compred} and the continuous embedding
 $\diff^k(S^1) \hookrightarrow \D^s(S^1), s \le k$.
\end{proof}

We do not know whether our local multiplier representations can be extended to a global multiplier
representation of $\widetilde{\D^s(S^1)}$. It is also open whether the global multiplier representation
of $\diff(S^1)$ with the Bott-Virasoro cocycle \cite[Proposition 5.1]{FH05} extends to $\widetilde{\D^s(S^1)}$
by continuity.

\section{Conformal nets and diffeomorphism covariance}\label{conformal}
Let $\mathrm{PSL}(2,\mathbb{R})$ be the M\"obius group and $\mathcal{I}$ be the set of nonempty, non-dense, open intervals of the unit circle $S^{1}$.
$I'$ denotes the interior of the complement of the interval $I\in\mathcal{I}$, namely $I'=(S^{1}\setminus I)^{\circ}$.
A {\bf M\"obius covariant net} $(\A, U, \Omega)$ on $S^{1}$ is a triple of a family 
$\mathcal{A}=\left\{\mathcal{A}(I), I\in\mathcal{I}\right\}$
of von Neumann algebras,
a strongly continuous unitary representation $U$ of $\mathrm{PSL}(2,\mathbb{R})$
acting on a separable complex Hilbert space $\mathcal{H}$ and $\Omega \in \H$,
satisfying the following properties:
\begin{enumerate}[{(}1{)}]
\item Isotony: $\mathcal{A}(I_{1})\subset\mathcal{A}(I_{2})$, if $I_{1}\subset I_{2}$, $I_{1},I_{2}\in \mathcal{I}$.
\item Locality: $\mathcal{A}(I_{1})\subset\mathcal{A}(I_{2})^{\prime}$, if $I_{1}\cap I_{2}=\emptyset$, $I_{1},I_{2}\in \mathcal{I}$.
\item M\"obius covariance: for $g\in \mathrm{PSL}(2,\mathbb{R})$, $I\in\mathcal{I}$,
\begin{equation*}
U(g)\mathcal{A}(I)U(g)^{-1}=\mathcal{A}(gI)
\end{equation*}
where $\mathrm{PSL}(2,\mathbb{R})$ acts on $S^{1}$ by M\"obius transformations.
\item Positivity of energy: the representation $U$ has positive energy, i.e. the conformal Hamiltonian $L_{0}$ (the generator of rotations) has non-negative spectrum.
\item Vacuum vector: $\Omega$ is the unique vector (up to a scalar) with the property
$U(g)\Omega=\Omega$ for $g\in \mathrm{PSL}(2,\mathbb{R})$.
Additionally $\Omega$ is cyclic for the algebra $\bigvee_{I\in\mathcal{I}}\mathcal{A}(I)$.
\setcounter{NET}{\value{enumi}}
\end{enumerate}
With these assumptions, the following automatically hold \cite[Theorem 2.19(ii)]{GF93}\cite[Section 3]{FJ96}
\begin{enumerate}[{(}1{)}]
\setcounter{enumi}{\value{NET}}
\item Reeh-Schlieder property: $\Omega$ is cyclic and separating for $\A(I)$.
\item Haag duality: for every $I\in\mathcal{I}$, $\mathcal{A}(I')=\mathcal{A}(I)'$ where $\mathcal{A}(I)'$ is the commutant of $\mathcal{A}(I)$.
\item Additivity: if $\lbrace I_{\alpha}\rbrace_{\alpha\in A}$ is a covering of $I\in\mathcal{I}$, with $I_{\alpha}\in\mathcal{I}$ for every $\alpha$, then ${\mathcal{A}(I)\subset \bigvee_{\alpha}\mathcal{A}(I_{\alpha})}$.
\item Semicontinuity: if $I_n\in\mathcal{I}$ is a decreasing family of intervals and $I=\left(\bigcap_n I_n\right)^{\circ}$ then \\$\mathcal{A}(I)=\bigwedge_n \mathcal{A}(I_n)$.
\setcounter{NET}{\value{enumi}}
\end{enumerate}

By a \textbf{conformal net} (or diffeomorphism covariant net) we shall mean a M\"obius covariant net which satisfies the following:
\begin{enumerate}[{(}1{)}]
\setcounter{enumi}{\value{NET}}
\item \label{diffcov} The representation $U$ extends to a projective unitary representation of $\diff(S^1)$
such that for all $I\in\mathcal{I}$ we have
\begin{align*}
U(\gamma)\mathcal{A}(I)U(\gamma)^*&=\mathcal{A}(\gamma I), \quad\gamma\in\diff(S^1),\\
U(\gamma)xU(\gamma)^*&=x,\hspace{3mm}x\in \mathcal{A}(I), \quad\gamma\in\diff(I^\prime)
\end{align*}
where $\diff(I^\prime)$ denotes the subgroup of diffeomorphisms $\gamma$ such that $\gamma(z)=z$ for all $z\in I$.
\end{enumerate}

A positive energy representation $U$ of $\diff(S^1)$ is equivalent to a direct sum of irreducible representations, see Proposition \ref{pr:compred}.
Every irreducible component $U_j$ in the decomposition has the same value of the central charge $c$ and if $h_j$ is the lowest weight of $U_j$,
$h_j-h_k\in\mathbb{Z}$ for every $j,k$. This fact is crucial for our purpose, which is to extend the conformal symmetry of the net
to the larger group $\D^s(S^1)$, $s>3$, in the sense that we want to show that the conditions in (\ref{diffcov}) are satisfied for arbitrary
$\gamma$ in $\D^s(S^1)$ and $\D^s(I')$ respectively.

\begin{proposition}
A conformal net $(\mathcal{A},U,\Omega)$ is $\D^s(S^1)$-covariant for every $s>3$.
\end{proposition}
\begin{proof}
Let $\lbrace\gamma_n\rbrace$ be a sequence of diffeomorphisms in $\diff(S^1)$ converging to $\gamma\in\D^s(S^1)$ in the topology of $\D^s(S^1)$
as in Lemma \ref{lem:localapprox}.
For all $n\in\mathbb{N}$ it holds that
\[
U(\gamma_n)\A(I)U(\gamma_n)^*=\A(\gamma_n I) \subset \A(\textstyle{\bigcup_{k=m}^n\gamma_k I}),
\]
where we used isotony of the net $\A$. For $x \in \A(I)$, it follows for $m \le n$ that
\[
U(\gamma_n)xU(\gamma_n)^*\in \A(\textstyle{\bigcup_{k=m}^n\gamma_k I}) = \bigvee_{k=m}^{\infty}\A(\gamma_k I),
\]
by additivity.
By Proposition \ref{cr:continuityB(H)}
it follows that $U(\gamma)xU(\gamma)^*=\lim_{n\rightarrow \infty} U(\gamma_n)xU(\gamma_n)^*$
(convergence in the strong operator topology) is in
$\bigcup_{k=m}^{\infty}\A(\gamma_k\cdot I)$ for any $m$, hence we have by upper semicontinuity that
\[
U(\gamma)\A(I)U(\gamma)^*\subset \bigcap_m\A(\textstyle{\bigcup_{k=m}^{\infty}\gamma_k I})= \A(\gamma I).
\]
The other inclusion follows by applying $\Ad U(\gamma^{-1})$.

Now consider $\gamma\in\D^s(I')$ and $x\in\A(I)$. We know from Lemma \ref{lem:localapprox}
that there exists a sequence $\lbrace \gamma_n\rbrace\subset \diff(I_n')$ converging to $\gamma$ in the topology of $\D^s(S^1)$
and a
decreasing sequence of intervals $I'_n\supset\supp(\gamma_n)\supset I'$ such that $\bigcap_n I'_n= I'$.
For $x\in\A(I_n)$,
$U(\gamma_m)xU(\gamma_m)^*=x$ if $m \ge n$, hence
by Proposition \ref{cr:continuityB(H)} we obtain
$U(\gamma)xU(\gamma)^*=x$.
As $n$ is arbitrary, this holds for any $x \in\A(\textstyle{\bigcup_n I_n}) = \A(I)$ by additivity.
\end{proof}

\subsection*{Representations of conformal nets}
Let $(\A,U,\Omega)$ a conformal net. A representation $\rho$ of $(\A,U,\Omega)$ is a family $\rho=\lbrace \rho_I\rbrace$, $I\in\I$, where $\rho_I$ are representations of $\A(I)$ on a common Hilbert space $\H_\rho$ and such that $\rho_J\vert_{\A(I)}=\rho_I$ if $I\subset J$. The representation $\rho$ is said to be locally normal if $\rho_I$ is normal for every $I\in\I$ (this is always true if the representation space $\H_\rho$ is separable \cite[Theorem 5.1]{TakesakiI}). We say that a representation $\rho$ of a conformal net $(\A,U,\Omega)$ is diffeomorphism covariant if there exists a positive energy representation $U^\rho$ of $\widetilde{\diff(S^1)}$  such that
\begin{equation*}
U^\rho(\gamma)\rho_I(x)U^\rho(\gamma)^*=\rho_{\mathring{\gamma}I}(U(\mathring{\gamma})xU(\mathring{\gamma})^*),\quad \text{ for } x\in\A(I),g\in\widetilde{\diff(S^1)},
\end{equation*}
where $\mathring{\gamma}$ is the image of $\gamma$ in $\diff(S^1)$ under the covering map.

Now let $\rho$ be a locally normal representation of the conformal net $\A$ and assume that $e^{i2\pi L^\rho_0}$ has pure point spectrum
(this is always the case if $\rho$ is a direct sum of irreducibles). By using \cite[Proposition 2.2]{Carpi04} and arguing as in the proof
of \cite[Proposition 3.7]{Carpi04} it is not hard to see that $\rho$ is diffeomorphism covariant (this will be directly proved in \cite{Tanimoto18-2})
and that the corresponding
positive energy projective unitary representation $U_\rho$ of $\widetilde{\diff(S^1)}$ is a direct sum of irreducibles.
By our previous results $U_\rho$ extends to  $\widetilde{\D^s(S^1)}$, $s>3$, and this extension makes $\rho$ $\widetilde{\D^s(S^1)}$-covariant. Furthermore, if $\rho$ is a direct sum of irreducible representations,
then the adjoint action $\Ad U_\rho(R(2\pi))$ is trivial, and in this sense
$\rho$ is $\D^s(S^1)$-covariant.
We summarize this fact in the following proposition. 

\begin{proposition}
Let $\rho$ be a locally normal representation of the conformal net $\A$ and
assume that $e^{i2\pi L^\rho_0}$ has pure point spectrum.
Then $\rho$ is $\widetilde{\D^s(S^1)}$-covariant for every $s>3$.
If further $\rho$ is a direct sum of irreducible representations, then
it is also $\D^s(S^1)$-covariant.
\end{proposition}

\section{Outlook}\label{outlook}
For all positive integers $n$ and some $h$, the irreducible unitary representation $U_{(n,h)}$
can be extended to $\D^s(S^1), s>2$ \cite{DIT19}.
It would be interesting to better understand to what extent the regularity of the diffeomorphisms can be weakened in such a way that the representations $U_{(c,h)}$ may be extended to such a class in a continuous way.
The proof of \cite{DIT19} (based on the strategy of \cite{Vromen13}) relies on the better-behaving $\mathrm{U}(1)$-current,
and it appears that such extensions do not act nicely on the stress-energy tensor $T$, which
we are currently able to extend only to $\mathcal{S}_\frac32(S^1)$.
On the other hand we know that at least some degree of regularity is required, i.e. we cannot just completely drop differentiability, at least when assuming that the representation has reasonable properties.
Indeed, using the modular theory of type $\mathrm{III}_1$ factors, it can be shown that a positive energy representation does not extend e.g. to the group of orientation
preserving homeomorphisms, still satisfying the locality property. For a detailed discussion on this point see \cite{DIT19}, in particular how this fact is related to the construction of soliton sectors for conformal nets.

Another interesting question is whether the global multiplier representations in \cite{FH05}
extend to $\widetilde{\D^s(S^1)}$. The question is whether these representations are continuous
in the $\D^s(S^1)$-topology.
Instead, what we used in Proposition \ref{pr:directsum} is the continuity of our extensions as projective representations,
and the existence of local multiplier representations follows.
In particular, we do not know whether there is a multiplier representation of $\widetilde{\D^s(S^1)}$
with the Bott-Virasoro cocycle.

\subsubsection*{Acknowledgements.}
S.C.\! would like to thank Gerard Misio\l{}ek for inspiring discussions. S.D.\! and S.I.\! would like to thank Stefano Rossi for the valuable discussions. Y.T.\! thanks Andr\'e Henriques and Kathryn Mann for useful information.
We are grateful to the referee for detailed comments.

S.C.\! ,  S.D.\! and Y.T.\! acknowledge the MIUR Excellence Department Project awarded
to the Department of Mathematics, University of Rome Tor Vergata, CUP E83C18000100006.

\appendix
\section{Appendix: Various definitions of Sobolev spaces}\label{norm}
\paragraph{By Fourier coefficients.}
In Section \ref{sobolev} we introduced for $s\geq 0$ the Sobolev spaces $H^s(S^1)$ through the Fourier coefficients,
 \begin{equation}
 H^s(S^1) := \{f \in L^2(S^1, \RR): \sum (1+k^2)^s|\hat f_k|^2 < \infty \}, \label{df:sobolev}
 \end{equation}
and for $s>\frac{3}{2}$ the Sobolev groups
\begin{equation}
 \D^s(S^1) := \{\g\in\diff^1(S^1): \tilde \g -\iota\in H^s(S^1)\}, \label{df:group}
\end{equation}
where $\tilde \g$ is a lift to $\widetilde{\diff^1}(S^1)$ and $\iota: \mathbb{R} \to \mathbb{R}$ is the identity map.
We can also give the topology first to $\widetilde{\D^s(S^1)}$ as an open subset of
$H^s(S^1) + \iota$ with the topology given by $H^s(S^1)$, then to $\D^s(S^1)$ by the quotient map.
The definition by Fourier coefficients is convenient for us because it is crucial that
$\D^{3+\e}(S^1)$ acts on the set of vector fields in $\mathcal{S}_{\frac32}(S^1)$ (to which the stress-energy tensor can be extended)
and the latter is defined again through Fourier coefficients.
However, we also cited Lemmas \ref{lm:sobolevalgebra}, \ref{lm:sobolevgroup}
from \cite[Lemma B.4, Theorem B.2]{IKT13} where the Sobolev spaces are defined in another way.
Therefore, we have to check that these definitions coincide.

\paragraph{By local integral.}
Let us first observe that an analogue of \cite[Lemma B.1]{IKT13} holds.
\begin{lemma}\label{lm:derivative}
 Let $f\in L^2(S^1, \RR)$ and $s>1$. Then, $f \in H^s(S^1)$ if and only if
 $f \in \dom(i\frac{d}{d\theta})$ and $f' \in H^{s-1}(S^1)$. Moreover, 
 the norm $\|f\| + \|f'\|_{H^{s-1}}$ is equivalent to $\|f\|_{H^s}$,
 where $\|f\| = \|f\|_{H^0}$ is the $L^2$-norm.
\end{lemma}
With this Lemma, we can consider the second characterization, parallel to \cite[Lemma B.2]{IKT13}.
\begin{lemma}\label{lm:norm}
 Let $s > 0, s \notin \ZZ$ and $\lambda = s - \sint$, where $\sint$ denotes the largest integer
 not exceeding $s$. Then $f \in H^s(S^1)$ if and only if
 $f \in H^{\sint}$ and $[f^{(\sint)}]_\lambda < \infty$, where
 $[f^{(\sint)}]_\lambda$ is the $L^2$-norm of the following function on $S^1\times S^1$
 \[
  (\theta_1,\theta_2) \longmapsto \frac{|f^{(\sint)}(e^{i\theta_1}) - f^{(\sint)}(e^{i\theta_2})|}{|\theta_1-\theta_2|^{\lambda + \frac12}},
 \]
 where $|\theta_1-\theta_2|$ denotes the distance\footnote{We may assume that
 $-\pi \le \theta_1, \theta_2 < \pi$, and $|\theta_2-\theta_1| = \min\{|\theta_2-\theta_1|, 2\pi - |\theta_2 - \theta_1|\}$,
 hence this depends only on $\theta_2 - \theta_2$.}
 of two points $\theta_1, \theta_2$ along the circle $S^1 = \RR/2\pi\ZZ$.
\end{lemma}
\begin{proof}
 Following \cite[Lemma B.2]{IKT13}, we prove the Lemma by induction.
 Let us assume $\sint = 0$, hence $s = \lambda$.
 We have
 \begin{align*}
  \int_{S^1}\int_{S^1}\frac{|f(e^{i\theta_1}) - f(e^{i\theta_2})|^2}{|\theta_1-\theta_2|^{2\lambda + 1}}d\theta_1 d\theta_2
  &=  \int_{S^1}\int_{S^1}\frac{|f(e^{i(\theta_1 +\theta)}) - f(e^{i\theta_1})|^2}{|\theta|^{2\lambda + 1}}d\theta_1 d\theta \\
  &=  \int_{S^1}\frac1{|\theta|^{2\lambda + 1}}\int_{S^1}|f(e^{i(\theta_1 +\theta)}) - f(e^{i\theta_1})|^2d\theta_1 d\theta. \\
 \end{align*}
 By Parseval's theorem,
 \begin{align*}
  \int_{S^1}|f(e^{i(\theta_1 +\theta)}) - f(e^{i\theta_1})|^2d\theta_1
  &= \sum_k |\widehat{f(e^{i(\cdot + \theta)})}_k - \hat f_k|^2 \\
  &= \sum_k |e^{ik\theta}-1|^2\cdot|\hat f_k|^2,
 \end{align*}
 therefore,
 \begin{align*}
  \int_{S^1}\int_{S^1}\frac{|f(e^{i\theta_1}) - f(e^{i\theta_2})|^2}{|\theta_1-\theta_2|^{2\lambda + 1}}d\theta_1 d\theta_2
  &=  \sum_k |\hat f_k|^2\int_{S^1}\frac{|e^{ik\theta}-1|^2}{|\theta|^{2\lambda + 1}} d\theta \\
  &=  \sum_k |k|^{2\lambda}|\hat f_k|^2\int_{S^1}\frac{|e^{ik\theta}-1|^2}{|k|^{2\lambda}|\theta|^{2\lambda + 1}} d\theta.
 \end{align*}
 Note that, for $k \neq 0$, by substitution $\hat \theta = k\theta$ we obtain
 \begin{align*}
  \int_{S^1}\frac{|e^{ik\theta}-1|^2}{|k|^{2\lambda}|\theta|^{2\lambda + 1}} d\theta
  &= \int_{-\pi}^{\pi}\frac{|e^{ik\theta}-1|^2}{|k|^{2\lambda}|\theta|^{2\lambda + 1}} d\theta \\
  &= \int_{-k\pi}^{k\pi}\frac{|e^{i\hat\theta}-1|^2}{|\hat\theta|^{2\lambda + 1}} d\theta,
 \end{align*}
 and this last integral is uniformly bounded both below and above with respect to $k$.
 Therefore, for $f \in H^{\sint}(S^1) = L^2(S^1, \RR)$, $\|f\|_{H^s} = \sum_k (1+|k|^2)^{\lambda}|\hat f_k|^2 < \infty$
 if and only if $\sum_k |k|^{2\lambda}|\hat f_k|^2 < \infty$, if and only if
 \begin{align*}
  \int_{S^1}\int_{S^1}\frac{|f(e^{i\theta_1}) - f(e^{i\theta_2})|^2}{|\theta_1-\theta_2|^{2\lambda + 1}}d\theta_1 d\theta_2
 < \infty.
 \end{align*}
 
 Assuming that the statement holds for $s-1$, we can conclude induction
 by applying it to $f'$ and using Lemma \ref{lm:derivative}.
 
\end{proof}
The whole Appendix B of \cite{IKT13} can be adapted to $H^s(S^1)$ and $\D^s(S^1)$ using these norms
and one obtains Lemma \ref{lm:sobolevgroup} for $s \notin \ZZ$, corresponding to \cite[Lemmas B.5,B.6]{IKT13}
(for $s \in \ZZ$, see \cite[Section 2]{IKT13}).
We believe this is the fastest way for the reader not familiar with Sobolev spaces.

\paragraph{By local Sobolev spaces.}
Alternatively, one may start with the Sobolev spaces on $\RR$, following \cite{IKT13}:
\[
 H^s(\RR, \RR) := \left\{f \in L^2(\RR, \RR): \int_\RR (1+\z^2)^s|\hat f(\z)|^2 d\z< \infty \right\},
\]
where $\hat f$ denotes the Fourier transform for $f\in L^2(\RR, \RR)$ (with a slight abuse of notation:
$\hat f$ depends on whether $f \in L^2(\RR, \RR)$ or $f \in L^2(S^1, \RR)$).
For an open connected set $\U \subset \RR$, we set (see \cite[Definition B.1]{IKT13},
where the boundary $\partial \U$ is required to be Lipschitz, but in $\RR$ it is not necessary):
\[
 H^s(\U, \RR) := \{f \in L^2(\U, \RR): \text{ there is } \tilde f\in H^s(\RR, \RR) \text{ s.t. } f= \tilde f|_\U\}.
\]
Let us consider $S^1$ as a manifold, namely supplied with an atlas $\{\U_k\}$.

For $s > \frac12$,
we may take another definition for $H^s(S^1) = H^s(S^1, \RR)$:
\[
 \{f \in C(S^1, \RR): \text{ for each }\theta \in S^1 \text{ there is } \U \ni \theta \text{ s.t.\! } f|_{\U} \in H^s(\U, \RR)\}.
\]
Now, from Lemma \ref{lm:norm}, it is clear that being in $H^s(S^1)$ is a local property (note that $S^1$ is compact).
If $f \in H^s(S^1)$ (in the sense of \eqref{df:sobolev}),
for any smooth function $\psi$ with compact support, $\psi f \in H^s(S^1)$ by Lemma \ref{lm:norm}.
To a chart $\U$ in the atlas, take a smooth function $\psi$ which is $1$ on $\overline \U$
and supported in a non-dense interval. Then $\psi f$ can be considered as an element of $H^s(\RR, \RR)$
by \cite[Lemma B.2]{IKT13}, hence the definition \eqref{df:sobolev} is stronger.
Conversely, if for each $\theta \in S^1$ there is $\U \ni \theta$ such that $f|_{\U} \in H^s(\U, \RR)$,
by compactness of $S^1$ one can take a finite cover $\{\U_k\}$ of $S^1$ and a smooth partition of unity $\{\psi_k\}$ subordinate to it,
and it follows that $f = \sum_k \psi_k f \in H^s(S^1)$ in the sense of \eqref{df:sobolev},
therefore, the two definitions are equivalent.

It is also clear that the following definition \cite[Section 4.3]{Taylor11vol1}
\[
 H^s(\U, \RR) := \{f \in L^2(\U, \RR): \psi f\in H^s(\RR, \RR) \text{ for any }\psi \in C^\infty(\U, \RR), \supp \psi \subset \U \}
\]
is equivalent to the definition through Fourier coefficients for $s \ge 0$.

Now, we define $H^s(S^1, S^1)$ to be the maps $f:S^1\to S^1$ such that
there are two atlases $\{\U_k\}, \{\V_k\}$ of $S^1$ such that
$f|_{\U_k} \in H^s(\U_k, \RR)$, where we identified $f(\U_k) \subset \V_k$ as a subset of $\RR$ by the chart
(see \cite[Section 3.1]{IKT13}) and
\[
 \D^s(S^1) = \{\g \in \diff^1(S^1): \g \in H^s(S^1, S^1)\}.
\]
Recall that the definition \eqref{df:group} is local,
and the identity map $\iota$ is a smooth function, hence it is equivalent to the definition above
which is manifestly local.

Now that we have the equivalence of definitions,
we can use \cite[Theorem B.2]{IKT13}, which we cited and specialized as Lemma \ref{lm:sobolevgroup}.

\paragraph{Local approximation of diffeomorphisms.}
We also need that elements in $\D^s(S^1)$ with compact support
can be approximated by elements $\D^s(S^1)$ with slightly larger support.

\begin{lemma}\label{lm:rotation}
 Let $s \ge 0$. For a fixed $f \in H^s(S^1)$, the rotation $\RR \ni t \mapsto f_t = f(e^{i(\cdot -t)}) \in H^s(S^1)$ is continuous.
\end{lemma}
\begin{proof}
 We have $\hat f_{t,k} = e^{ikt}\hat f_k$,
 and hence $|\hat f_{t,k}| = |\hat f_k|$ and $\hat f_{t,k} \to \hat f_k$ as $t\to 0$.
 By Lebesgue's dominated convergence theorem (applied to the measure space $\ZZ$ with the counting measure,
 with the dominating function $k \mapsto 4|(1+k^2)^s\hat f_k|^2$)
 \[
  \sum_k(1+k^2)^s |\hat f_{t,k} - \hat f_k|^2 \to 0.
 \]
 This means $\|f-f_t\|_{H^s} \to 0$.
\end{proof}

\begin{lemma}\label{lem:localapprox}
Let $s \ge \frac32$.
For every $\g\in\D^s(S^1)$, there exists a sequence $\lbrace \g_n\rbrace \subset \diff(S^1)$ converging to $\g$ in the topology of $\D^s(S^1)$.
Furthermore, if $\g$ is supported in $I$, we can take $\g_n$ such that
$\supp \g_n \supset \g_{n+1}$ and $\bigcap_n \supp \g_n = I$.
\end{lemma}
\begin{proof}
Let $\gamma\in\D^s(S^1)$ and $\varphi \in \widetilde{\D^s(S^1)}$
such that $\varphi(\theta+2\pi)=\varphi(\theta)+2\pi$ and $\gamma(e^{i\theta})=e^{i\varphi(\theta)}$.
If $\g$ is supported in a proper interval we may assume without loss of generality that $\varphi(\theta)=\theta$ if $\theta\in[-\pi,a)\cup(b,\pi]$.
The function $\psi\coloneqq \varphi'-1$ is $2\pi$-periodic and has compact support $[a,b]$ as a function on $[-\pi,\pi]$.

We now choose a set of $C^{\infty}$-functions $\lbrace g_n\rbrace$ with compact support strictly contained
in $[-\pi,\pi]$ such that for all $n\in\mathbb{N}$ $g_n\geq 0$, $\int g_n=1$, $\supp(g_n)\supset\supp(g_{n+1})$,
$\supp(g_n)\rightarrow \lbrace 0\rbrace$. In addition, if $\g$ is supported in $[a,b]$, we may assume that
$[a,b]+\supp(g_n)\supset \supp(\psi*g_n)$,
where the convolution is defined on $S^1 \simeq \RR/2\pi\ZZ$ as an abelian group.
With this choice, $\psi*g_n + \iota$ defines an element in $\D^s(S^1)$,
because $\int_t \psi*g_n (t)dt = 0$ and $\psi * g_n > 0$ for sufficiently large $n$.

To obtain the claim, it is enough to show that $\|\psi - \psi*g_n\|_{H^s} \to 0$
as $n\to 0$.
This follows from
\[
 \|\psi - \psi*g_n\|_{H^s} \le \int_{S^1} g_n(t)\|\psi - \psi_t\|_{H^s} dt
\]
and Lemma \ref{lm:rotation}.

\end{proof}

{\small
\def\cprime{$'$} \def\polhk#1{\setbox0=\hbox{#1}{\ooalign{\hidewidth
  \lower1.5ex\hbox{`}\hidewidth\crcr\unhbox0}}} \def\cprime{$'$}

}

\end{document}